\documentclass[11pt,a4paper]{article}
\usepackage[T1]{fontenc}
\usepackage[utf8]{inputenc}
\usepackage[margin=1.3in]{geometry}
\usepackage{tabularx, enumerate,enumitem, authblk,soul}
\usepackage[dvipsnames]{xcolor}
\usepackage{amsmath,amssymb,dsfont,subcaption}
\usepackage{mathtools, graphicx,color,dsfont,amsthm,bbm, mathrsfs}
\usepackage{lipsum,bm,comment}
\usepackage{natbib}
\usepackage{longtable}
\usepackage{array}
\usepackage{caption}
\usepackage[colorlinks=true,linkcolor=blue,anchorcolor=blue,citecolor=blue,filecolor=black,menucolor=black,runcolor=black,urlcolor=black]{hyperref}
\usepackage{cleveref}
\theoremstyle{plain}
\newtheorem{Theor}{Theorem}[section]
\newtheorem{Lem}[Theor]{Lemma}

\newtheorem{Prop}[Theor]{Proposition}
\newtheorem{Corol}[Theor]{Corollary}

\theoremstyle{definition}

\newtheorem{Example}[Theor]{Example}
\newtheorem{Examples}[Theor]{Examples}
\newtheorem{Remark}[Theor]{Remark}

\newcommand{\B}{\mathbb{B}}
\newcommand{\C}{\mathbb{C}}

\newcommand{\E}{\mathbb{E}}

\newcommand{\I}{\mathbb{I}}

\newcommand{\N}{\mathbb{N}}

\renewcommand{\P}{\mathbb{P}}

\newcommand{\R}{\mathbb{R}}
\renewcommand{\S}{\mathbb{S}}

\newcommand{\W}{\mathbb{W}}

\newcommand{\CC}{\mathcal{C}}
\newcommand{\DD}{\mathcal{D}}

\newcommand{\HH}{\mathcal{H}}

\newcommand{\LL}{\mathcal{L}}

\newcommand{\NN}{\mathcal{N}}

\newcommand{\TT}{\mathcal{T}}

\newcommand{\FFF}{\mathscr{F}}

\newcommand{\g}{{\rm g}}
\newcommand{\Fpg}{F_P^\g}
\newcommand{\Qpg}{Q_P^\g}

\def\1{{\mathbb I}}
\newcommand{\Ind}[1]{\1\left[#1\right]}
\newcommand{\ie}{\textit{i.e.} }

\newcommand{\ps}[2]{\left\langle #1,#2 \right\rangle}
\newcommand{\sm}{\setminus}

\DeclareMathOperator{\diag}{diag}

\DeclareMathOperator*{\argmin}{arg\,min}

\newcommand{\loc}{{\rm loc}}
\newcommand{\spa}{\quad\quad}

\newcommand{\wto}{\rightharpoonup}

\let\emptyset\varnothing

\usepackage{scalerel,stackengine}
\stackMath
\newcommand\wwidehat[1]{%
	\savestack{\tmpbox}{\stretchto{%
			\scaleto{%
				\scalerel*[\widthof{\ensuremath{#1}}]{\kern-.6pt\bigwedge\kern-.6pt}%
				{\rule[-\textheight/2]{1ex}{\textheight}}%WIDTH-LIMITED BIG WEDGE
			}{\textheight}% 
		}{0.5ex}}%
	\stackon[1pt]{#1}{\tmpbox}%
}

\newcommand{\RN}[1]{%
	(\textup{\uppercase\expandafter{\romannumeral#1}})%
}
\newcolumntype{H}{>{\setbox0=\hbox\bgroup}c<{\egroup}@{}}

\newenvironment{Proof}[1]
{
	\begin{proof}[Proof of #1]
	}
	{
	\end{proof}
}

\allowdisplaybreaks

\title{Spatial depth characterizes probability measures}

\author{Alberto Gonz{\'a}lez-Sanz$^*$ and Dimitri Konen$^{\dag}$
}
\affil{$^*$Columbia University, $^{\dag}$University of Cambridge
} 
\date{}

\begin{document}

\maketitle

% \begin{frontmatter}

% \title{On the characterization property of spatial quantiles and depth in finite and infinite dimension}
% %\title{A sample article title with some additional note\thanksref{T1}}
% \runtitle{On the characterization property of spatial quantiles and depth in infinite dimension}
% %\thankstext{T1}{A sample of additional note to the title.}

% \begin{aug}
% \author[A]{\fnms{Alberto}~\snm{Gonz{\'a}lez-Sanz}\ead[label=e1]{ag4855@columbia.edu}}
% \author[B]{\fnms{Dimitri}~\snm{Konen}\ead[label=e2]{dk738@cam.ac.uk}}
% \and
% \author[C]{\inits{}\fnms{Davy}~\snm{Paindaveine}\ead[label=e3]{Davy.Paindaveine@ulb.be}}

% \address[A]{Department of Statistics, Columbia University\printead[presep={,\ }]{e1}}

% \address[B]{
% Department of Pure Mathematics and Mathematical Statistics, 
% University of Cambridge, 
% Cambridge, CB3 0WA,
% United Kingdom\printead[presep={,\ }]{e2}}

% \address[C]{ECARES and Department of Mathematics, 
% Universit\'{e} libre de Bruxelles,
% Avenue F.D. Roosevelt, 50, CP114/04,
% B-1050, Brussels,
% Belgium\printead[presep={,\ }]{e3}}
% \end{aug}

\begin{abstract}
We solve two open problems about spatial (or geometric) quantiles and depth. First we show that in infinite dimension, the spatial distribution function and the associated spatial quantiles characterize the underlying distribution, which has been established in \cite*{Kol1997} in finite dimension but remained unknown in infinite dimension.
% ; this is achieved through a combination of results from complex analysis and Fourier transforms. 
Second, and more surprisingly, we show that the spatial depth also fully characterizes probability measures, which has been an open problem even in finite dimension since the introduction of these concepts in \cite*{Cha1996} and \cite*{VarZha2000}.
% , and our proof crucially relies on gradient flow techniques. 
Our results provide theoretical foundations for  nonparametric depth-based statistical inference and introduce novel proof techniques to investigate these questions for other depth and quantile concepts.
\end{abstract}

% \setcounter{tocdepth}{1}
%{
%\hypersetup{linkcolor=black}
%\tableofcontents
%}

%  \vspace{0.5em}

% {\small
% \noindent 
% \emph{Keywords} 
% Breakdown point; {Multivariate quantiles}; {Geometric quantiles}; {Spatial quantiles}; {Spatial depth}; {Statistical depth}

% \noindent \emph{AMS 2020 Subject Classification}  62G05; {62R10}; {62G30}
% }
%  \vspace{.0em}

%\end{frontmatter}

\section{Introduction} \label{secIntro} 
Let $(\HH, \langle \cdot, \cdot \rangle)$ be a separable real Hilbert space with norm $\|\cdot\|$; throughout, $\HH$ may be finite- or infinite-dimensional. Strong (or norm) convergence is denoted by $\to$ and weak convergence by $\wto$. The spatial (or geometric) distribution function (cdf) of an arbitrary probability measure $P$ over $\HH$ is defined in \cite*{Kem1987} and \cite*{Cha1996} as the map 
\begin{equation}\label{defFpg}
\Fpg : \HH \mapsto \overline{\mathbb{B}}
,\ 
x\mapsto \int_{\HH\sm\{x\}} \frac{x-z}{\|x-z\|}\, dP(z) 
,
\end{equation}
where $\overline{\B}\equiv \{x\in \HH : \|x\|\leq 1\}$ is the closed unit ball of $\HH$; note that the $\HH$-valued integral is defined (uniquely) through the Riesz representation theorem. When $\HH=\R$, the distribution function in (\ref{defFpg}) coincides with the usual univariate cdf $F_P$; specifically we have $\Fpg(x) = 2F_P(x)-1$ for all $x\in\R$. In arbitrary dimensions, the cdf $\Fpg$ is invertible under suitable assumptions on $P$, so that its inverse 
$
\Qpg \equiv (\Fpg)^{-1}
:
\overline \B \to \HH
$
is a natural extension of the notion of quantile map to multivariate and functional settings. Associated with these concepts is the \emph{spatial depth}, defined as 
\begin{equation}\label{defSD}
{\rm SD}(x;P)
\equiv 
1-\|\Fpg(x)\|
,\spa 
\forall\ x\in\HH 
,
\end{equation}
which measures the `depth' (equivalently, $1-{\rm SD}(x;P)$ measures the outlyingness) of any location $x\in\HH$ with respect to $P$: A depth close to $0$ indicates that $x$ is far from the bulk of the distribution whereas a depth close to $1$ indicates that $x$ is central for~$P$; see also \cite{Nag2017}. In this paper, we are concerned with understanding when the spatial objects---cdf, quantile, and depth functions---characterize the underlying probability distribution. \smallskip 

These spatial concepts have met a great success and are commonly used in practice (\cite*{Mottonen97, Cha2003, PaiVanB2012, GirStu2015, ChoCha19, DaoStuUss24}) because of their conceptual simplicity and their robustness properties (\cite*{LopRou1991, KonPai2025, KonPai2025a}), and since they can be computed very fast even in high dimensions (\cite*{Frietal2012}). In addition, spatial quantiles are naturally defined in infinite-dimensional Hilbert spaces and even in general Banach spaces (see, e.g., \cite*{Romon2022, PasReid2022}) or manifolds (\cite*{KonPai23}), while other location functionals, because they exploit intrinsically finite-dimensional features of the ambient space, are limited to Euclidean spaces.
% ; this is the case, for instance, with the cdf and quantile notions based on optimal transport (\cite*{HallEspagne21, HalKon24}).  \textcolor{purple}{This is not true \cite*{gonzalezsanz.etl.a.2025monotone} defines OT-based quantiles in general Hilnert spaces }

\smallskip

One of the most important properties of $\Fpg$ in Euclidean spaces, \ie when $\HH=\R^d$, is that it fully characterizes $P$; this provides the basis to design rank-based statistical procedures. The characterization property was first established in \cite*{Kol1997} (see Theorems~2.5 and 2.9 there) 
% through a Fourier transform approach 
by noticing that $\Fpg(x)$ is a convolution of $P$ with the kernel $K(x)=x/\|x\| \I[x\ne 0]$, where $\I$ denotes the indicator function. Their strategy consists in computing the exact form of the Fourier transform $\FFF(K)$ of $K$ and in showing that $\FFF(K)$ is a Borel map on $\R^d\sm\{0\}$ such that $\FFF(K)\neq 0$ almost everywhere. From there, one deduces that the equality $\Fpg(x) = F_Q^\g(x)$ for all $x\in\R^d$ implies $\FFF(K)( \FFF(P) - \FFF(Q)) = 0$ to the effect that $\FFF(P)=\FFF(Q)$ hence also $P=Q$. This argument, however, fails when $\dim \HH=\infty$ since $\FFF(K)$ crucially depends on the dimension $d$ and does not admit a natural limit as $d\to \infty$ (see also \cite*{Kon2025}). In the context of nonparametric testing based on  `distance covariances', \cite*{Lyons.2013.AoP} showed that separable Hilbert spaces are of `strong negative type', \ie for any finite signed measure $\mu$, if $\int \|x-z\|\, d\mu(z)=0$ for all $x$ then $\mu=0$ provided $\mu$ admits first order moments. Under suitable assumptions, the Fréchet differential in $x$ of the previous integral functional is $F_\mu(x)\equiv \int K(x-z)\, d\mu(z)$, so that $\mu=0$ if $F_\mu$ vanishes identically. Their proof is based on a Gaussian variant of the Crofton embedding and is not adaptable to general measures without moment assumptions. We thus engineer a new proof strategy, inspired by ideas in \cite*{GorKol87}, that applies to arbitrary probability measures and is based on an asymptotic coordinate perturbation of $F_\mu$ combined with Carlson's Theorem from complex analysis and Fourier Transform techniques, to establish the following result in general Hilbert spaces:

\begin{Theor}\label{TheorCdf}
    Let $P$ and $Q$ be Borel probability measures on $\HH$. (i) If $\Fpg(x)=F_Q^\g(x)$ for all $x\in\HH$, then $P=Q$. (ii) If $\Fpg(x)=F_Q^\g(x)$ for all $x$ in a dense subset $\DD\subset \HH$, then $P=Q$.
\end{Theor}

In particular, when $\Fpg$ is invertible, the quantile map $\Qpg$ thus also characterizes~$P$. Whether the spatial depth in (\ref{defSD}) characterizes $P$ is a much more delicate question, as none of the proof strategies employed so far apply to this object due to the presence of the norm. %\com{This property has even been thought to be wrong (REF ?).} 
To address this question, we devise a new method inspired by \cite*{PerezSalasSavid-Math-Prog.2021} where, perhaps surprisingly, it is shown that a convex function is fully determined by its minimal-norm subgradient at all points. We leverage this fact by observing that $\Fpg$ is the differential of the convex map
\begin{equation}\label{defHp}
    h_P(x) 
    \equiv 
    \int_{\HH} \big\{ \|x-z\| - \|z\| \big\}\, dP(z)
    ;
\end{equation}
subtracting $\|z\|$ in the integral allows to make sense of $h_P$ without moment assumptions on $P$. Even when $h_P$ is not differentiable, the point $\Fpg(x)\in\HH$ still belongs to the subgradient $(\partial h_P)(x)$ of $h_P$ at $x$. However, $\Fpg(x)$ need not be of minimal norm within $(\partial h_P)(x)$ so that the result alluded to above cannot be applied as such. We thus adapt ideas in \cite*{PerezSalasSavid-Math-Prog.2021} to our setting and evaluate the difference $\Fpg - F_Q^\g$ along a suitable gradient flow, while taking care of  issues at atoms of $P$.
% , which allows us to prove the following:

\begin{Theor}\label{TheorDepth}
    Let $P$ and $Q$ be Borel probability measures on $\HH$.
    \begin{itemize}
    \item[i)] If $\|\Fpg(x)\|=\|F_Q^\g(x)\|$ for all $x\in\HH$, then $P=Q$.
    
    \item[ii)] If $\|\Fpg(x)\|=\|F_Q^\g(x)\|$ for all $x$ in a dense subset $\DD\subset \HH$ and $\dim \HH\ge 2$, then $P \ne Q$ only if there exists $m_P,m_Q\in\HH$ distinct and $\beta\in (1/2,1]$ such that 
    $$
    P
    =
    \beta \delta_{m_P} + (1-\beta) \delta_{m_Q}
    ,\quad 
    \text{and}\quad
    Q
    =
    (1-\beta) \delta_{m_P} + \beta \delta_{m_Q}
    ,
    $$
    where $\delta_x$ denotes the Dirac probability measure at $x$. 

    \item[iii)] If $\|\Fpg(x)\|=\|F_Q^\g(x)\|$ for all $x$ in a dense subset $\DD\subset \HH$ and $\dim \HH = 1$, then $P \ne Q$ only if $P$ and $Q$ have unique medians $m_P$ and $m_Q$ (resp.), $P$ and $Q$ coincide on the Borel subsets of $\HH\sm\{m_P,m_Q\}$, they give no mass to the open interval joining $m_P$ and $m_Q$, and $P[\{m_P\}]>0$ and $Q[\{m_Q\}]>0$.
    \end{itemize}
\end{Theor}

This result is striking as it establishes that any probability measure
% , initially characterized by a whole `vector field' $\Fpg$, 
is in fact entirely determined by the scalar-valued field $\{\|\Fpg(x)\| : x\in\HH\}$, thus showing that the directional information contained in $\Fpg$ is superfluous. 
% Similar results are difficult to establish for other depth functionals proposed in the literature, or may even be known to fail such as for the Tukey (or halfspace) depth (see~\cite*{Nag2021}). Nevertheless, 
The ideas underlying the proof of \Cref{TheorDepth} are not specific to the case of spatial depth; these techniques 
% (or a suitable modification thereof) 
apply more generally to any functional expressed as the differential of a convex map. In particular, results similar to \Cref{TheorDepth} hold for the depth concept based on measure transportation from \cite*{HallEspagne21}; see also \Cref{remark:OT-charact} below. Theorem~\ref{TheorDepth} is a strong and important result, as similar statements are known to fail for other depth functionals proposed in the literature; see, e.g., \cite*{Nag2021} on the failure of the characterization for Tukey's (or halfspace) depth.\smallskip 
 %Our proof strategy, which is based on some recent results on convex analysis \citep{PerezSalasSavid-Math-Prog.2021}, applies to the transport-based depth function (not only the quantiles) for finite and infinite dimensional spaces, see \Cref{remark:OT-charact}. 
% The first contribution of this paper is to provide an affirmative answer to both  both open problems. Our main result, \Cref{TheorCaracterisation}, shows that if $P,Q\in \PPP(\HH)$  are such that $\|\Fpg(x)\|=\|F_Q^\g(x)\|$ for all $x\in\HH$,  the $P=Q$. This obviously implies that the spatial depth function  characterizes the underlying probability measure.   
% This is an important property for the spatial depth  which fails (or is unknown) for most of the depth functionals proposed in the literature  (see e.g.~\cite*{Nag2021} for the Tukey depth).  Our proof strategy, which is based on some recent results on convex analysis \citep{PerezSalasSavid-Math-Prog.2021}, applies to the transport-based depth function (not only the quantiles) for finite and infinite dimensional spaces, see \Cref{remark:OT-charact}. 

Finally, we investigate the mapping properties and regularity of the spatial cdf $\Fpg$. In finite dimension, the map $\Fpg$ defines a homeomorphism between the whole space and its open unit ball under mild assumptions (see, e.g., Theorem~6.2 in \cite*{KonPai1} with $\rho(t)\equiv t$).
% ; similar results were also established in related settings involving the more general $\rho$-quantiles (\cite*{KonPai1}) and regularized spatial quantiles (\cite*{KonStu2025}). 
Below, we extend these results to the infinite-dimensional case. The regularity of the cdf $\Fpg$ and quantile map $\Qpg$ were studied in \cite*{Kon2025} in finite dimension. It is shown there that, in dimension $d$, if $P$ admits a density $p\in L^\infty(\R^d)$ then $\Fpg\in C^{d-1}(\R^d)$ but no more; obtaining further regularity can only be achieved by requiring additional smoothness for $p$ (see Proposition~4.2 there for a counter-example). In infinite dimension, one can conjecture that if the densities of all finite-dimensional marginals are bounded, then $\Fpg$ is smooth by formally taking $d=\infty$ above. We now establish this rigorously under weaker assumptions: 
% The second contribution of this paper involves some analytical properties of $\Fpg$ for non-singular $P$. First we show that if $P$ is non-atomic and not supported on a single line of $\HH$, then the quantile and distribution functions are inverses of each other. This complements some known results in the literature (see  \cite*{Kem1987,Romon2022,KonPai2025}).  Furthermore, under some mild integrability conditions, we show that $\Fpg$ defines a diffeomorphism between $\HH$ and the open unit ball $\mathbb{B}$. 

\begin{Theor}\label{TheorInvertible}
Let $P$ be a Borel probability measure on $\HH$. Assume that there exists a linear subspace $V \subset \HH$ of finite dimension~$d\ge 2$, with corresponding orthogonal projection $\Pi_V:\HH\to V$, such that the pushforward $\Pi_V \# P$ admits a (non-identically zero) density with respect to the Lebesgue measure on $V$ that is bounded on each compact subset of~$V$. Then, $\Fpg$ defines a $C^\infty$-diffeomorphism between $\HH$ and its open unit ball. In particular, all quantile contours of $P$ are $C^\infty$-diffeomorphic to the unit sphere of $\HH$.
% In addition, its inverse $\Qpg$ is continuous from $(\B, \|\cdot\|)$ to $(\HH, \|\cdot\|)$.
% weak-to-weak continuous, \ie $\Qpg(\alpha_k u_k)\wto \Qpg(\alpha u)$  when $\B\supset(\alpha_k u_k)\wto \alpha u\in\B$, and strong-to-strong continuous, \ie $\Qpg(\alpha_k u_k)\to \Qpg(\alpha u)$ when $\alpha_k u_k\to \alpha u$.
% $\Qpg: \mathbb{B}\to \HH$ defines a $\mathcal{C}^\ell$-diffeomorphism with inverse $\Fpg$.  
\end{Theor}
% The proof of Theorem~\ref{TheorInvertible}  is given in Section~\ref{Section:soothbes}, where we establish a slightly more general regularity result allowing finite-dimensional marginals with suitable local integrability, beyond the locally bounded case stated here. 

Surprisingly, Theorem~\ref{TheorInvertible}  only requires that there exists at least one finite dimensional marginal with a bounded density to guarantee smoothness of $\Fpg$. These results complement existing literature (\cite*{Kem1987,Romon2022, PasReid2022, KonPai2025}) on spatial quantiles in infinite dimension. 

\section{Proof of Theorem \ref{TheorCdf}} \label{sec:ProofsCdf}

Since $\DD$ is dense in $\HH$, then the equality $\Fpg(x)=F_Q^\g(x)$ holds at each $x\in\HH$ that is a continuity point of both $\Fpg$ and $F_Q^g$. The next lemma precisely characterizes  this set.

\begin{Lem}\label{LemCont}
    Let $P$ be a Borel probability measure on $\HH$. Then $\Fpg$ is continuous at $x\in\HH$ if and only if $P[\{x\}]=0$.
\end{Lem}

\begin{Proof}{Lemma \ref{LemCont}}
    Fix an arbitrary sequence $(x_k)\subset \HH$ such that $x_k\neq x$ for all $k$ and $x_k\to x$. Then, we have 
    $$
    \Fpg(x_k)
    -
    \Fpg(x) 
    =
    \frac{x_k-x}{\|x_k-x\|}P[\{x\}]
    +
    \int \Big\{ \frac{x_k-z}{\|x_k-z\|}\I[z\neq x_k] - \frac{x-z}{\|x-z\|} \Big\} \I[z\neq x]\, dP(z) 
    .
    $$
    The dominated convergence theorem entails that the integral in the last display converges to $0$, so that $\Fpg(x_k)\to \Fpg(x)$ if and only if $P[\{x\}]=0$.
\end{Proof}
\vspace{2mm}

We denote the collection of atoms of $P$ and $Q$ by 
$$
A_{P,Q}
\equiv
\{x \in \HH : P[\{x\}] + Q[\{x\}] > 0\}
.
$$
Because $\DD$ is dense in $\HH$, and since $\Fpg$ and $F_Q^\g$ are both continuous on $\HH\sm A_{P,Q}$, we deduce that 
\begin{equation}\label{eq:APQ}
\Fpg(x)=F_Q^\g(x)
,\spa 
\forall\ x\in\HH\sm A_{P,Q}
.
\end{equation}
We now distinguish the finite- and infinite-dimensional cases. \smallskip 

Assume first that $\HH$ is finite-dimensional, $\HH=\R^d$, say. Because $A_{P,Q}$ is at most countable, then $A_{P,Q}$ has Lebesgue measure $0$, so that $\Fpg$ and $F_Q^\g$ coincide (Lebesgue) almost everywhere. In particular, they define the same Schwartz distribution on $\R^d$; see, e.g., Section~A.1.1 in \cite*{Kon2025Supp} for further details on the use of distributions (generalized functions) in this setting. In particular,   Theorem~3.2 in \cite*{Kon2025} entails that $P=Q$ on the Borel subsets of $\R^d$. This yields the conclusion of Theorem~\ref{TheorCdf} when $\HH$ is finite-dimensional.\smallskip

In the remainder of this section, we thus assume that $\HH$ is infinite-dimensional. Because $\HH$ is separable, let $\{e_k : k\ge 0\}$ be a (countable) orthonormal basis of $\HH$. For the argument that will follow, we will need to choose this basis in such a way that some of its translates and dilates do not intersect $A_{P,Q}$. The fact that one such basis exists is established in the next lemma. As is common in this setting, the proof makes use of the axiom of choice through Zorn's lemma.

\begin{Lem}\label{lemma:nice-basis}
    For all $x\in \HH$ and $c>0$, there exists a countable orthonormal basis $\{e_k : k\ge 0\}$ of $\HH$ such that  $\{x+c e_k : k\ge 0\} \subset \HH\sm A_{P,Q}$.
\end{Lem}

\begin{Proof}{Lemma \ref{lemma:nice-basis}}
By replacing $A_{P,Q}$ with $(A_{P,Q}-x)/c$, we may assume without loss of generality
that $c=1$ and $x=0$. Define $\mathcal{O}$ as the set of all countable orthonormal systems contained in $ A_{P,Q}^c= \HH\sm A_{P,Q}$. This set is nonempty as $A_{P,Q}^c\cap \S\neq \emptyset$ follows from the fact that $A_{P,Q}$ is countable. We define the order relation given by the inclusion $U\leq V$ if $U\subseteq V$. Then every ordered chain $\{A_\alpha\}_{\alpha} \subseteq \mathcal{O}$ can be upper bounded by $\bigcup_\alpha A_\alpha$, which is an orthonormal system with all elements in $A_{P,Q}^c$. By Zorn's lemma there exists a maximal element $U\in \mathcal{O}$. If $U$ is a basis of $\HH$, the claim follows. Assume that $U$ is not a basis of $\HH$. Then $(\overline{{\rm span}(U)})^\perp$ is closed, nonempty and
$$
 \HH= \overline{{\rm span}(U)} \oplus  (\overline{{\rm span}(U)})^\perp.
$$
If ${\rm dim}((\overline{{\rm span}(U)})^\perp)>1$, then, as $A_{P,Q}$ is at most countable, it follows that
\[
A_{P,Q}^c\cap (\overline{{\rm span}(U)})^\perp \cap \S\neq \emptyset .
\]
Hence, for any $y\in A_{P,Q}^c\cap (\overline{{\rm span}(U)})^\perp \cap \S$, the set $U\cup \{y\}$ belongs to $\mathcal{O}$ and it is strictly larger than $U$; a contradiction. If ${\rm dim}((\overline{{\rm span}(U)})^\perp)=1$, then $(\overline{{\rm span}(U)})^\perp$ is generated by some vector $x_0 \in \S$. We pick $x_1\in U\subset A_{P,Q}^c$. Then it follows that
$$
 \HH= \overline{{\rm span}(U \setminus\{x_1\})} \oplus  {\rm span}(\{x_0,x_1\}).
$$
We claim that ${\rm span}(\{x_0,x_1\}) \cap A_{P,Q}^c$ contains at least two orthonormal vectors $\{y_1,y_2\}$. Note that in this case
\[
   (U \setminus\{x_1\}) \cup \{y_1,y_2\}
\]
is an orthonormal basis of $\HH$ contained in $A_{P,Q}^c$. To show the claim we identify the space ${\rm span}(\{x_0,x_1\})$ with $\R^2$, and ${\rm span}(\{x_0,x_1\})\cap A_{P,Q}^c$ with a subset $A$ of $\R^2$ having countable complement. Let $\S^1\subset \R^2$ be the unit circle. We denote by $(a_1,a_2)$ the coordinates of a vector $a$ in an orthonormal basis $\{e_1,e_2\}$ of $\R^2$. We define the rotation $R_{\frac{\pi}{2}}$ of angle $\pi/2$ as
$$
 R_{\frac{\pi}{2}}(a_1,a_2)=(-a_2,a_1).
$$
Note that $\{R_{\frac{\pi}{2}}(a),a\}$ is an orthonormal basis for all $a\in\S^1$. Then it is enough to show that
\[
R_{\frac{\pi}{2}}\left(A\cap \S^1\right)\cap A
\]
is nonempty. This follows from the fact that $R_{\frac{\pi}{2}}\left(A\cap \S^1\right)$ is uncountable, so that it cannot be contained in the complement of $A$. Therefore, we have found an orthonormal basis of $\HH$ contained in $\HH\sm A_{P,Q}$. Since $\HH$ is separable, this basis is countable.
\end{Proof}
\vspace{2mm}

Before proceeding with the proof, we state a result from complex analysis known as Carlson's Theorem that will be needed in the remainder of the proof. We denote by $\Re(\zeta)$ and $\Im(\zeta)$ the real and imaginary components of $\zeta\in \C$. 

% With these preparations, we can state the following lemma.

\begin{Lem}\label{LemAnalytic}
Let $E=\{\zeta\in\C : \Re(\zeta)\ge 0\}$ and $f:E\to \C$ be a continuous function. Assume the following: (i) $f$ is analytic on $\Re(\zeta)>0$, (ii) for some $\alpha \in \R$ and a constant $A>0$, we have $|f(\zeta)|\le A e^{\alpha |\zeta|}$ for all $\Re(\zeta)>0$, and (iii) for some $\beta \in (-\infty, \pi)$ and a constant $B>0$, we have $|f(\zeta)|\le B e^{\beta |\zeta|}$ for all $\Re(\zeta)=0$. If $f(n)=0$ for all natural numbers $n=0,1,2,\ldots$, then $f(\zeta)=0$ for all $\zeta\in E$.
\end{Lem}

Lemma \ref{LemAnalytic} does not correspond to Carlson's actual result (see Theorem~C of §9, p.~58, in \cite*{CarlesonThesis}), but it is a fairly straightforward consequence of it. The version provided in Lemma \ref{LemAnalytic} is based on §5.8 and §5.81, pp.~185--186, in \cite*{Tit1939}; see also Theorem 1 in \cite*{Hardy1920}. It is usually stated in terms of holomorphic maps that are also regular which, for simplicity, we enforced in Lemma~\ref{LemAnalytic} by imposing additional continuity up to the boundary of $E$ (see, e.g., p.~143 in \cite*{Tit1939}), which will be enough for our purposes. \smallskip 

We now proceed with the rest of the proof of Theorem \ref{TheorCdf}, assuming that (\ref{eq:APQ}) holds and that $\HH$ is infinite-dimensional with orthonormal basis $(e_k)$ as in Lemma~\ref{lemma:nice-basis}. We divide the proof in several steps.\medskip

\newpage 

{\it Step 1. Coordinate perturbation at infinity.}\smallskip

\noindent 
Fix $x,z\in \HH$ and $c>0$. Let $\{e_k\}_{k\in \N}$ be the orthonormal basis provided in \Cref{lemma:nice-basis}. Then we have 
$$
\|z-x-ce_k\|^2 
=
\|z-x\|^2 - 2c \ps{z-x}{e_k} + c^2 
.
$$
Since Parseval's identity yields $\|z-x\|^2 = \sum_{k\geq 0} \lvert \ps{z-x}{e_k}\rvert^2$, we have $\ps{z-x}{e_k}\to 0$ as $k\to\infty$. We deduce that
$$
\|z-x-c e_k\| 
\to 
\sqrt{\|z-x\|^2 + c^2}\
>
0, \quad \text{as $k\to \infty$.}
$$
Consequently, since $x+ce_k\in \HH\setminus A_{P,Q}$,  the dominated convergence theorem entails that
$$
\ps{\Fpg(x+ce_k)}{e_k}
=
\int \frac{\ps{x+ce_k-z}{e_k}}{\|x+ce_k-z\|}\, dP(z) 
\to 
\int \frac{c}{\sqrt{\|x-z\|^2+c^2}}\, dP(z),
$$
as $k\to \infty$.  Similarly, we have
$$
\big\langle F_Q^\g(x+ce_k), e_k\big\rangle
\to 
\int \frac{c}{\sqrt{\|x-z\|^2+c^2}}\, dQ(z) 
,
$$
as $k\to \infty$. 
Since $\Fpg$ and $F_Q^\g$ coincide over $\HH\sm A_{P,Q}$ and $\{x+ce_k : k\ge 0\}\subset \HH\sm A_{P,Q}$, we have $\Fpg(x+ce_k)=F_Q^\g(x+ce_k)$ for all $k$. We deduce that 
\begin{equation}\label{eq:LimitF2}
\int \frac{1}{\sqrt{\|x-z\|^2+c}}\, dP(z) 
=
\int \frac{1}{\sqrt{\|x-z\|^2+c}}\, dQ(z) 
% =
% \int_{\HH} \Big( \sqrt{ \|z-x\|^2 + c } - \|z\| \Big)\, dQ(z) 
,\spa x\in \HH,\ \forall\ c> 0
.
\end{equation}
% In addition, the triangle inequality entails that 
% $$
% \big| \|z-x-c e_k\| - \|z\| \big| 
% \le 
% \|x-ce_k\| 
% \le 
% \|x\| + |c|
% ,
% $$
% for all $k\ge 0$ and $z\in \HH$. Consequently, the dominated convergence theorem yields 
% $$
% h_P(x+c e_k)
% =
% \int_{\HH} \big( \|z-x-c e_k\| - \|z\| \big)\, dP(z) 
% \to 
% \int_{\HH} \Big( \sqrt{ \|z-x\|^2 + c^2 } - \|z\| \Big)\, dP(z) 
% ,\spa 
% k\to \infty 
% .
% $$
% Since the same result holds for $Q$, and $h_P(x+c e_k)=h_Q(x+ce_k)$ for all $x\in \HH$, we deduce that (\ref{eq:LimitF}) holds for all $x\in \HH$ and $c\in\R$.
\medskip 

{\it Step 2. Extension to the complex plane.} \smallskip 

\noindent
For all $x,z\in \HH$ fixed, we have 
$$ \frac{d}{dc} \frac{1}{(\|z-x\|^2+c)^{\frac{1}{2}}}=-\frac{1}{2(\|z-x\|^2 + c)^{\frac{3}{2}}}
.
$$ 
For any $c>0$, the latter is bounded, uniformly in $c\in [c_0,\infty)$, by $2^{-1}c^{-3/2}$. Repeating the argument, we may take derivatives under the integral in (\ref{eq:LimitF2}) to obtain, for any integer $n\ge 0$,
\begin{equation}\label{eq:negative-powers-eq}
\int \frac{1}{(\|z-x\|^2 + c)^{\frac12 + n}}\, d(P-Q)(z) 
=
0
% =
% \int_{\HH} \frac{1}{(\|z-x\|^2 + c)^{\frac12 + n}}\, dQ(z) 
,\spa 
\forall\ x\in \HH,\ c > 0
.
\end{equation}
Now fix $c>0$ and $x\in \HH$. We denote by $P_{c,x}$ and $Q_{c,x} $ the laws in $\R$ of $\log(\|Z-x\|^2+c)$ when $Z\sim P$ and $Z\sim Q$, respectively. Then  \eqref{eq:negative-powers-eq} yields
\begin{equation}
\int e^{-sn} e^{-s/2}\, d(P_{c,x}-Q_{c,x})(s)
=
0
% =
% \int_\R e^{-sn} e^{-s/2}\, dQ_{c,x}(s)
.
\end{equation}
Since $P_{c,x}$ and $Q_{c,x}$ are supported on $[\log c,\infty)$, we may define
$$
\Phi(\zeta) 
\equiv
\int e^{-s\zeta} e^{-s/2}\, d(P_{c,x}-Q_{c,x})(s)
,\spa 
\forall\ \zeta\in \C,\ \Re(\zeta)\ge -\frac12
.
$$
It is straightforward to see that $\Phi(\zeta)$ is holomorphic on
$
\{\zeta \in \C : \Re(\zeta) > -1/2 \}
.
$
Consequently, $\Phi$ is analytic over this complex half-plane.
% In particular, $\Phi$ is analytic over $\Re(\zeta)\ge 0$.
\medskip 

{\it Step 3. Conclusion through Fourier transform.}\smallskip 

\noindent
Observe that $\Phi$ is continuous over $\Re(\zeta)\ge 0$ and analytic over $\Re(\zeta)>0$. Recalling that $P_{c,x}$ and $Q_{c,x}$ are supported on $[\log c,\infty)$, we have
\begin{align*}
|\Phi(\zeta)|
\le 
\int |e^{-s\Re(\zeta)}|\ e^{-s/2} d(P_{c,x}+Q_{c,x})(s)
\le 
2 e^{-(\log c) \Re(\zeta)} e^{-(\log c)/2}
\le 
\frac{2}{ \sqrt{c} } e^{|\log c|\ |\zeta|}
,
\end{align*}
for all $\zeta\in\C$ with $\Re(\zeta)>0$, as well as 
$$
|\Phi(\zeta)|
\le 
\int e^{-s/2} d(P_{c,x}+Q_{c,x})(s)
=
2 e^{-(\log c)/2} 
\le 
\frac{2}{ \sqrt{c} }
,
$$
for all $\zeta\in\C$ with $\Re(\zeta)=0$. Since $\Phi(n)=0$ for all natural numbers $n\ge 0$, we deduce from Lemma \ref{LemAnalytic} that $\Phi(\zeta)=0$ for all $\Re(\zeta)\ge 0$. Taking $\zeta = it$, $t\in\R$, thus leads to 
$$
0
=
\Phi(it) 
=
\int e^{-its} e^{-s/2} d(P_{c,x}-Q_{c,x})(s) 
,\spa 
\forall\ t\in\R 
.
$$
This means that the Fourier transform (or characteristic function) of the measures $e^{-s/2}dP_{c,x}(s)$ and $e^{-s/2} dQ_{c,x}(s)$ coincide. This implies that $P_{c,x}$ and $Q_{c,x}$ coincide over Borel subsets of $\R$. Indeed, letting $p$ and $q$ stand for the density of $P_{c,x}$ and $Q_{c,x}$, respectively, with respect to the finite measure $\mu = P_{c,x}+Q_{c,x}$, entails that $e^{-s/2} p(s) = e^{-s/2} q(s)$ for $\mu$-almost every $s\in\R$. Since $e^{-s/2}>0$ for all $s\in\R$, we deduce that $p(s)=q(s)$ for $\mu$-almost every $s\in\R$, so that $P_{c,x}=Q_{c,x}$ over Borel subsets of $\R$. Consequently, we have, for all $x\in \HH$, $c>0$, and $r\ge \log c$,
\begin{align*}
P\big[\{z\in \HH : \|z-x\| \le \sqrt{e^r - c}\} \big]&= P_{c,x}\big[(-\infty,r]\big]\\
&=Q_{c,x}\big[{(-\infty,r]}\big]
\\
&=
Q\big[\{z\in \HH : {\|z-x\| \le \sqrt{e^r - c}}\}\big].
\end{align*}
Setting $r(u)=\log(u^2+c)$ for $u\ge 0$, we have
\begin{equation}\label{eq:PQBalls}
P\big[\{z\in \HH : \|z-x\| \le u\} \big]
=
Q\big[\{z\in \HH : \|z-x\| \le u\}\big]
,\spa 
\forall\ x\in \HH,\ u\ge 0
.
\end{equation}
We deduce that $P$ and $Q$ coincide over all balls of $\HH$ centered at arbitrary points in $\HH$. Since $\HH$ is separable, then it is a Lindel\"of space, so that, in particular, every open set can be written as a countable union of balls of finite radius. This implies that the $\sigma$-algebra generated by such balls coincides with the Borel $\sigma$-algebra of $\HH$. It follows that $P$ and $Q$ coincide over Borel subsets of $\HH$. This concludes the proof. 
% \end{Proof}

\section{Proof of Theorem~\ref{TheorDepth}}
\label{Section:characterization}

When $\|\Fpg(x)\|=\|F_Q^\g(x)\|$ for all $x$ in a dense subset of $\HH$, then arguing as in the beginning of the proof of Theorem \ref{TheorCdf} entails that $\|\Fpg(x)\|=\|F_Q^\g(x)\|$ for all $x\in \HH\sm A_{P,Q}$, where $A_{P,Q}=\{y\in\HH : P[\{y\}]+Q[\{y\}]>0\}$. In this section, we first prove Theorem~\ref{TheorDepth}(ii)-(iii) under the previous assumption before finally proving Theorem~\ref{TheorDepth}(i) in Proposition~\ref{PropPEqualQDim1} below.\medskip

Fix $x_0\in\HH$ and define 
$$
\Phi(x)\equiv h_P(x)+h_Q(x)
,\spa 
\forall\ x\in\HH 
.
$$
Note that $\Phi$ is convex and continuous on $\HH$. Since $h_P$ and $h_Q$ are continuous on $\HH$ and coercive, in the sense that $h_P(x), h_Q(x) \to + \infty$ as $\|x\|\to \infty$, then Proposition~11.15 in \cite*{Bauschke-Combettes.Springer.2017} entails that $h_P$ and $h_Q$ admit a global minimum. In particular, we have
$$
\inf_{x\in\HH} \Phi(x) \geq \inf_{x\in\HH} h_P(x)+ \inf_{x\in\HH} h_Q(x) >-\infty
.
$$ 
Consequently, \Cref{PropGradFlow} entails that there exists a unique map $u:[0,\infty)\to \R^d$ satisfying 
\begin{equation}\label{eq:GradFlow}
\begin{cases}
-\dot u(t) \in (\partial \Phi)(u(t))\quad  \text{for almost all}~ t>0,\\[2mm]
u(0) = x_0.
\end{cases}
\end{equation}
Let us now give the idea of the proof in the specific case when $\Phi$, $h_P$ and $h_Q$ are differentiable and $\Phi$ is strictly convex. Then, $u(t)$ is the solution to the gradient flow equation $\dot u(t) = - \nabla \Phi(u(t))$ with $u(0)=x_0$, and taking derivatives in $t$ yields 
% The heuristics is the following
\begin{align*}
\frac{d}{dt} \big( h_P(u(t)) - h_Q(u(t)) \big) 
&=
\big\langle\Fpg(u(t)) -F_Q^\g(u(t)) , \dot u(t) \big\rangle 
\\[2mm]
&=
-\big\langle\Fpg(u(t)) -F_Q^\g(u(t)) , \Fpg(u(t)) + F_Q^\g(u(t)) \big\rangle 
\\[2mm]
&=
\|F_Q^\g(u(t))\|^2 -\|\Fpg(u(t))\|^2
.
\end{align*}
Therefore, when $\|\Fpg\|=\|F_Q^\g\|$ over $\HH$, it follows that $t\mapsto h_P(u(t))-h_Q(u(t))$ is constant over $(0,\infty)$; this constant further depends on the starting point $x_0$. However, when $\Phi$ is strictly convex then the gradient flow $u(t)$ converges to the unique minimizer $x_\Phi$ of $\Phi$ as $t\to\infty$, which is now independent of $x_0$. Thus, taking limits yields 
$$
h_P(x_0)-h_Q(x_0)
=
\lim_{t\to\infty} 
\big(h_P(u(t))-h_Q(u(t)\big)
=
h_P(x_\Phi)-h_Q(x_\Phi)
,
$$
from which we deduce that $h_P-h_Q$ is in fact constant over all of $\HH$. In addition, this constant must be zero since $h_P(0)=h_Q(0)=0$, from which we deduce that $h_P=h_Q$ hence $\Fpg = F_Q^\g$ by taking the differential, which yields $P=Q$. \smallskip  

\begin{Remark}[Characterization through transport-based depth]\label{remark:OT-charact}
The measure transportation distribution function $\mathbb{F}_P$ of $P$ is defined as the unique gradient of a convex function pushing $P$ forward to a fixed reference probability measure $\mu$ which is assumed to have bounded support; see, e.g., \cite*{HallEspagne21} and \cite*{gonzalezsanz.etl.a.2025monotone}. The transport-based distribution function $\mathbb{F}_P$ is only defined $P$-a.e., but it can be extended as the element of minimal norm of a maximal monotone extension of $\mathbb{F}_P$, which we still denote by $\mathbb{F}_P$. Such a choice satisfies the assumptions in  \cite*{PerezSalasSavid-Math-Prog.2021} so that that the transport-based depth function
% $$ x\mapsto {\rm D}^{OT}(x;P) \equiv  1-\|\mathbb{F}_P(x)\|$$
characterizes $P$. 
\end{Remark}

The proof strategy above only applies under some assumptions on $h_P$ and $h_Q$, which are granted when $P$ and $Q$ are non-atomic and not supported on a same line. The proof we present below is more involved and encompasses \textit{all} probability measures $P$ and $Q$.

\begin{Lem}\label{LemAbsCont}
Let $P$ and $Q$ be Borel probability measures on $\HH$, and $u(t)$ be the gradient flow from (\ref{eq:GradFlow}). The maps $t\mapsto h_P(u(t))$ and $h_Q(u(t))$ are absolutely continuous on every interval of $[0,\infty)$, and for almost all $t>0$ we have 
$$
\frac{d}{dt}\big(h_P(u(t))\big)
=
\langle \Fpg(u(t)),\dot{u}(t) \rangle
,\quad
\text{and} 
\quad 
\frac{d}{dt}\big(h_Q(u(t))\big)
=
\langle F_Q^\g(u(t)),\dot{u}(t) \rangle
.
$$
\end{Lem}

\begin{Proof}{Lemma \ref{LemAbsCont}}
Let us first argue that $\Fpg(x)\in (\partial h_P)(x)$ for all $x\in\HH$. The subgradient $(\partial h_P)(x)$ is the collection of points $y\in\HH$ such that 
$$
h_P(x) + \ps{y}{z-x} \le h_P(z)
,\spa 
\forall\ z\in\HH
.
$$
Parametrizing $z=z(t,v)$ as $z=x+tv$ for $t\ge 0$ and $v\in\S$, and using the convexity of~$h_P$ to the effect that the map $t\mapsto (h_P(x+tv)-h_P(x))/t$ is non-decreasing on~$(0,\infty)$, we see that 
$$
(\partial h_P)(x)
=
\Big\{ y\in\HH : \ps{y}{v} \le \frac{\partial h_P}{\partial v}(x)~~ \forall\ v\in\S 
\Big\} 
,
$$
where $\partial h_P/\partial v$ is given (see Proposition~3.4 in \cite*{KonStu2025} with $\mathfrak{r}\equiv 1$ and $\alpha=0$) by
% , by virtue of (\ref{eq:JacR}) for $x\in\R^d\sm\{0\}$,
$$
\frac{\partial h_P}{\partial v}(x)
\equiv
\lim_{s\downarrow 0} \frac{h_P(x+sv)-h_P(x)}{s}
=
\|v\|P[\{x\}] + \ps{\Fpg(x)}{v}
.
% =
% \begin{cases}
%     \ps{\Fpg(x)}{v} & \text{if } x\ne 0, \\
%     h_P(0) & \text{if } x=0.
% \end{cases}
% =
% \begin{cases}
%     h_P(\|x\|) \frac{\ps{x}{v}}{\|x\|} & \text{ if } x\ne 0, \\
%     h_P(0) & \text{ if } x=0.
% \end{cases}
$$
We deduce that
\begin{equation}\label{eq:SubGradHp}
(\partial h_P)(x)
=
\Big\{ 
y\in\HH : \big\|y - \Fpg(x)\big\| \le P[\{x\}] 
\Big\} 
,
% =
% \begin{cases}
% \{ (\nabla h_P)(x)\} & \text{if } x\ne 0,\\
% B(0, h_P(0))\cup\{0\} & \text{if } x=0.
% \end{cases}
\end{equation}
which yields $\Fpg(x)\in (\partial h_P)(x)$ for all $x\in\HH$. Similarly, we have $F_Q^\g(x)\in (\partial h_Q)(x)$ for all $x\in\HH$. Consequently, \Cref{PropGradFlow}(v) with $\psi=h_P$ and $f(t)=\Fpg(u(t))$, and $\psi=h_Q$ and $f(t)=F_Q^\g(u(t))$, respectively, implies that $h_P(u(t))$ and $h_Q(u(t))$ are absolutely continuous on every interval of $[0, \infty)$ with almost everywhere defined weak derivatives as in the statement.
\end{Proof}

The following result shows that $h_P - h_Q$ is constant along the flow $u(t)$. Next we will show that this constant does not depend on the starting point $x_0$. 
% \medskip 

\begin{Lem}\label{lemma:Constant} 
Let $P$ and $Q$ be Borel probability measures on $\HH$, and $u(t)$ be the gradient flow from (\ref{eq:GradFlow}). Define 
$
T_\Phi
\equiv
\inf\big\{ 
t>0 : u(t)\in \argmin\Phi\big\} 
,
$ with the convention $T_\Phi = \infty$ if the set is empty.
Then, the map $t\mapsto h_P(u(t)) - h_Q(u(t))$ is constant over $[0, T_\Phi)$. 
\end{Lem}

\begin{Proof}{Lemma \ref{lemma:Constant}}
Assume first that $u(0)\in \argmin \Phi$. \Cref{PropGradFlow}(iv) entails that $t\mapsto \Phi(u(t))$ is non-increasing so that, since $u(0)\in \argmin \Phi$, then $\Phi(u(t))$ is in fact constant over $(0,\infty)$. Consequently, \Cref{PropGradFlow}(iv) yields $\dot u(t)=0$ for almost all $t>0$. Lemma \ref{LemAbsCont} thus provides 
$$
\frac{d}{dt}\big(h_P(u(t)) - h_Q(u(t))\big) 
=
\big\langle \Fpg(u(t))-F_Q^\g(u(t)), \dot u(t)\big\rangle
=
0
,
$$
so that $t\mapsto h_P(u(t)) - h_Q(u(t))$ is constant.\smallskip

Then assume that $x_0\in\HH\setminus\argmin\Phi$. We first show that the set of
times at which the flow visits an atom of $P+Q$ before reaching $\argmin\Phi$ is
at most countable. Notice that $T_\Phi>0$ since $t\mapsto u(t)$ is continuous, and that the value $T_\Phi=\infty$ is allowed. Set
\begin{align*}
    S&\equiv
\Big\{ 
t\in (0,T_\Phi) : P[\{u(t)\}]+Q[\{u(t)\}]>0
\Big\}  .
\end{align*}
% Then, it follows from (\ref{eq:SubGradHp}) that 
% \begin{equation}
%     \label{eq:Complement-of-S}
% (0,T_\Phi)\setminus S\subset \Big\{ 
% 		t\in (0,T_\Phi) :  \text{$\partial h_P(u(t))=\{ \Fpg(u(t))\} =\partial h_Q(u(t))=\{ \mathbb{F}_P^g(u(t))\} $}
% 		\Big\}.
% \end{equation}
Let us now argue that $S$ is at most countable. Assume, \emph{ad absurdum}, that $S$ is uncountable. Since $P+Q$ is a finite measure, it has at most countably many atoms. In particular, there exists $z\in\HH$ such that $P[\{z\}]+Q[\{z\}]>0$ and such that the set $S_z\equiv\{t\in (0,T_\Phi) : u(t)=z\}$ is uncountable. Let $\tau_-\equiv\inf S_z$ and $\tau_+\equiv\sup S_z$. Since $S_z$ contains more than two elements, we have $\tau_-<\tau_+\le T_\Phi$. Because $t\mapsto \Phi(u(t))$ is non-increasing by virtue of \Cref{PropGradFlow}(iv), we have 
$$
\Phi(z)
=
\Phi(u(\tau_+))
\leq 
\Phi(u(t))
\leq 
\Phi(u(\tau_-))=\Phi(z)
,\spa 
\forall\ t\in [\tau_-,\tau_+]
.
$$
It follows that $\Phi(u(t))$ is constant on $[\tau_-,\tau_+]$ which, by Proposition~\ref{PropGradFlow}(iv), implies that  $\dot{u}(t)=0$ for almost all $t\in [\tau_-,\tau_+]$ and, thus, $u(t)=z$ for all $t\in [\tau_-,\tau_+]$. Recalling that $-\dot{u}(t)\in  \partial \Phi (u(t))$ for almost all $t>0$, we deduce that $0\in  \partial \Phi (z)$ which, by convexity of $\Phi$, yields $z\in \argmin \Phi$. The fact that $u(\tau_-)=z$ with $z\in\argmin \Phi$ thus contradicts the fact that $\tau_-<T_\Phi$. We deduce that $S$ is at most countable. \smallskip

For all $t\in (0, T_\Phi)\sm S$, hence for almost all $t\in (0,T_\Phi)$, we then have
$$
 -\dot{u}(t)\in \partial \Phi (u(t))
=
\big\{\Fpg(u(t)) + F_Q^\g(u(t)) \big\}
.
$$
Recalling that $\|\Fpg(x)\|=\|F_Q^\g(x)\|$ for all $x\in \HH$ such that $P[\{x\}]+Q[\{x\}]=0$, we find that, for all $t\in (0,T_\Phi)\sm S$
\begin{align*}
\frac{d}{dt} \big( h_P(u(t)) - h_Q(u(t)) \big) 
&=
-\big\langle\Fpg(u(t)) -F_Q^\g(u(t)) , 
\Fpg(u(t))+F_Q^\g(u(t)) \big\rangle
\\[2mm]
&=
-\|\Fpg(u(t))\|^2+\|F_Q^\g(u(t))\|^2
=
0
,
\end{align*} 
which concludes the proof. 
\end{Proof}

We introduce the median sets of $P$ and $Q$ as the collection of their spatial (or Fréchet) medians: 
$$
{\rm med}(P)
\equiv 
\Big\{
x_0\in \HH : x_0~ \text{is a global minimizer of}~ h_P 
\Big\}
,
$$
and we define ${\rm med}(Q)$ is a similar fashion. By continuity and convexity of $h_P$, the set ${\rm med}(P)$ is a closed convex subset of $\HH$.\smallskip

We first prove \Cref{TheorDepth} in a special case.

\begin{Lem}\label{LemNotInLine}
Let $P$ and $Q$ be Borel probability measures on $\HH$, and $u(t)$ be the gradient flow from (\ref{eq:GradFlow}). Assume that $P$ and $Q$ are not supported on a same line, or that either ${\rm med}(P)\subset {\rm med}(Q)$ or ${\rm med}(Q)\subset {\rm med}(P)$ holds. If $\|\Fpg(x)\|=\|F_Q^\g(x)\|$ for all $x$ belonging to a dense subset $\mathcal{D}\subset \HH$, then $P=Q$.
\end{Lem}

\begin{Proof}{Lemma \ref{LemNotInLine}}
     Fix $x_0\in\DD$ and let $u(t)$ be the corresponding gradient flow from (\ref{eq:GradFlow}). Since $h_P(u(t))$ and $h_Q(u(t))$ are absolutely continuous, \Cref{lemma:Constant} yields
		$$
		h_P(u(t))-h_Q(u(t)) 
		=
		h_P(x_0)-h_Q(x_0)
		,
		$$
		for all $t\in (0,T_\Phi)$. Since $\Phi$ is convex, continuous and coercive, then $\argmin\Phi$ is a nonempty closed convex set. In particular, if $T_\Phi<\infty$ then $u(T_\Phi)\in \argmin\Phi$. If $T_\Phi=\infty$, then Proposition~\ref{PropGradFlow}~(iii) and~(vi)  entail that 
        $\Phi(u(t))\to   \min \Phi $ and $u(t)\rightharpoonup u_\infty$ for some $u_\infty\in\argmin\Phi$. Since $h_P$ and $h_Q$ are weakly lower semi-continuous  \citep[Theorem~9.1]{Bauschke-Combettes.Springer.2017}, we have
        \begin{align*}
            \limsup_{t\to  \infty} h_P(u(t)) -  h_Q(u(t))&= \limsup_{t\to  \infty} \Phi(u(t)) - 2 h_Q(u(t))\\
            &\leq \Phi(u_\infty) - 2 h_Q(u_\infty)=   h_P(u_\infty)-h_Q(u_\infty) 
        \end{align*}
and
\begin{align*}
\liminf_{t\to\infty}\big(h_P(u(t))-h_Q(u(t))\big)
&=
\liminf_{t\to\infty}\big(2h_P(u(t))-\Phi(u(t))\big)\\
&\ge
2h_P(u_\infty)-\Phi(u_\infty)=
h_P(u_\infty)-h_Q(u_\infty).
\end{align*}
        Hence, irrespective of whether $T_\Phi$ is finite or infinite, there exists $u_\infty \in \argmin \Phi$ such that
        $$
		h_P(u_\infty)-h_Q(u_\infty) 
		=
		h_P(x_0)-h_Q(x_0)
		.
		$$
		Assume for now that the quantity $h_P(u_\infty)-h_Q(u_\infty)$ is independent of $x_0$. Then the map $x\mapsto h_P(x)-h_Q(x)$ is constant over $\DD$. The fact that $\DD$ is dense in $\HH$ and that $h_P$ and $h_Q$ are continuous over $\HH$ implies that $x\mapsto h_P(x)-h_Q(x)$ is constant over~$\HH$. Since $ h_P(0)=h_Q(0)=0$, we deduce that $h_P(x) = h_Q(x)$ for all $x\in\HH$. Taking directional derivatives as in the proof of Lemma~\ref{LemAbsCont} thus yields 
        $$
        \langle\Fpg(x) - F_Q^\g(x), v\rangle 
        =
        Q[\{x\}]-P[\{x\}]
        ,\spa\forall\ v\in \S,~ \forall\ x\in \HH
        .
        $$
        Since the r.h.s. in the last display is independent of $v$, we must have $\Fpg(x) = F_Q^\g(x)$ for all $x\in\HH$. Consequently, Theorem~\ref{TheorCdf} yields $P=Q$.\smallskip 
        
        It remains to show that the value of $h_P(u_\infty)-h_Q(u_\infty)$ is independent of $x_0$. First assume that $P$ and $Q$ are not supported on a same line. Then letting $\mu\equiv(P+Q)/2$, Theorem~1 in \cite*{PaiVir2021} entails that $h_\mu = \Phi/2$ admits a unique minimizer over $\HH$, which thus coincides with $u_\infty$. In particular, the value of $h_P(u_\infty)-h_Q(u_\infty)$ is independent of $x_0$. Assume now, without loss of generality, that ${\rm med}(P)\subset {\rm med}(Q)$, \ie $\argmin h_P\subset \argmin h_Q$. In particular, we have $\argmin h_P\subset \argmin \Phi$ so that $\inf\Phi =\inf h_P + \inf h_Q$, where all $\argmin$ and $\inf$ are taken over $\HH$. Since $u_\infty \in \argmin\Phi$, we thus have 
		$$
		h_P(u_\infty)
		=
		\inf \Phi - h_Q(u_\infty) 
		\leq \inf \Phi - \inf h_Q 
		=
		\inf h_P 
		,
		$$
		so that $h_P(u_\infty)=\inf h_P$. A similar reasoning yields $h_Q(u_\infty)=\inf h_Q$. In particular, the value of $h_P(u_\infty)-h_Q(u_\infty)$ is equal to $\inf h_P - \inf h_Q$ and is therefore independent of $x_0$. This concludes the proof.
\end{Proof}
%   We now conclude the proof of Theorem~\ref{TheorDepth} in the case where $P$ and
% $Q$ are not supported on a same line. Assume that $\|\Fpg(x)\|=\|F_Q^\g(x)\|$ for all $x$ belonging to a dense subset $\mathcal{D}\subset \HH$. By Lemma~\ref{LemNotInLine},
% $h_P=h_Q$ on $\HH$.  Hence $\Fpg(x)=F_Q^\g(x)$ for every $x\in\HH$.
% Theorem~\ref{TheorCdf} therefore implies $P=Q$.    \qed

\begin{Lem}\label{LemPdiffQForm}
    Let $P$ and $Q$ be Borel probability measures on $\HH$, and assume that $\|\Fpg(x)\|=\|F_Q^\g(x)\|$ for all $x$ in a dense subset $\DD\subset \HH$. If $P\ne Q$, then $P$ and $Q$ are supported on a same line $\LL$, they have a unique median, $m_P\in\LL$ and $m_Q\in\LL$, respectively, with $m_P\ne m_Q$ and $P[\{m_P\}]>0$ and $Q[\{m_Q\}]>0$. In addition, $P$ and~$Q$ coincide on the Borel subsets of $\HH\sm\{m_P,m_Q\}$, and we have 
    $$
    % P\big[\LL\sm [m_P,m_Q]\big]=Q\big[\LL\sm [m_P,m_Q]\big]
    % ,\quad
    % \text{and}\quad
    P[(m_P,m_Q)] 
    =
    Q[(m_P,m_Q)]
    =0
    ,
    $$
    where $(m_P, m_Q)$ 
    % and $[m_P,m_Q]$ 
    denotes the open line segment in $\HH$ joining $m_P$ and $m_Q$.
\end{Lem}

\begin{Proof}{Lemma \ref{LemPdiffQForm}}
    Assume that $P\ne Q$ so that $P$ and $Q$ are supported on a same line by virtue of Lemma~\ref{LemNotInLine}, and denote by $\LL=\{z_0 + \lambda v : \lambda\in\R\}$ the common supporting line of $P$ and $Q$. Denote by $\Lambda_P$ a random variable such that $z_0+\Lambda_P v$ has law $P$, and by $\Lambda_Q$ the corresponding random variable for $Q$. Observe that 
    $$
    \Fpg(z_0+\lambda v) 
    = 
    \big(\P(\Lambda < \lambda) - \P(\Lambda > \lambda)\big) v
    =
    \big(2F_{\Lambda_P}(\lambda)-1-\Lambda_P[\{\lambda\}]\big)v
    ,
    $$
    where $F_{P_\lambda}$ stands for the usual univariate cdf of $\Lambda_P$. The same computation yields a similar identity for $F_Q^\g(z_0+\lambda v)$ in terms of $\Lambda_Q$. We observed at the beginning of Section~\ref{Section:characterization} that $\|\Fpg(x)\|=\|F_Q^\g(x)\|$ for all $x$ such that $P[\{x\}]+Q[\{x\}]=0$, which provides
    \begin{equation}\label{eq:CdfUni}
    \big\lvert 2F_{\Lambda_P}(\lambda)-1-\Lambda_P[\{\lambda\}] \big\rvert
    =
    \big\lvert 2F_{\Lambda_Q}(\lambda)-1-\Lambda_Q[\{\lambda\}] \big\rvert
    ,
    \end{equation}
    for all $\lambda\in\R$ such that $\Lambda_P[\{\lambda\}]+\Lambda_Q[\{\lambda\}]=0$. Theorem~1(iv) in \cite*{PaiVir2021} entails that ${\rm med}(P)$ and ${\rm med}(Q)$ are both subsets of $\LL$ and thus coincide with ${\rm med}(\Lambda_P)$ and ${\rm med}(\Lambda_Q)$, respectively. Since they are closed and convex, they are of the form ${\rm med}(\Lambda_P)\equiv [m_P^-, m_P^+]$ and ${\rm med}(\Lambda_Q)\equiv [m_Q^-, m_Q^+]$, where $m_P^- = m_P^+$ and $m_Q^-=m_Q^+$ are both allowed. Also note from (\ref{eq:SubGradHp}) or (14) in \cite*{KonStu2025} that any $x_0\equiv z_0+\lambda_0v\in {\rm med}(P)$ is characterized by $\|\Fpg(x_0)\|\le P[\{x_0\}]$ which, in our case, rewrites as 
    \begin{equation}\label{eq:CharacCdfUni}
    \frac12 \le F_{\Lambda_P}(\lambda_0)\le \frac12 + \Lambda_P[\{\lambda_0\}]
    .
    \end{equation}
    \smallskip 

    We start by establishing that $m_P^- = m_P^+$ and $m_Q^-=m_Q^+$. Assume, \textit{ad absurdum}, that $m_P^-<m_P^+$. Since the set of atoms of $\Lambda_P$ is at most countable, we have $F_{\Lambda_P}(\lambda)=1/2$ for almost all $\lambda\in [m_P^-,m_P^+]$. Because $F_{P_\lambda}$ is monotone, we then have $F_{\Lambda_P}(\lambda)=1/2$ for all $\lambda\in (m_P^-, m_P^+)$ and, in fact, for all $\lambda\in [m_P^-, m_P^+)$ by right-continuity. Also note that $|2F_{\Lambda_P}(\lambda)-1|=|2 F_{\Lambda_Q}(\lambda)-1|$ for almost every $\lambda\in\R$. In particular, we have $|2F_{\Lambda_Q}(\lambda)-1|=0$ for almost all $\lambda\in [m_P^-, m_Q^+]$ which, by the same reasoning as before, yields $F_{\Lambda_Q}(\lambda)=1/2$ for all $\lambda\in [m_P^-, m_P^+)$. It follows that $[m_P^-, m_P^+)\subset {\rm med}(Q)$ which, since ${\rm med}(Q)$ is closed, yields 
    $$
    {\rm med}(P)=[m_P^-, m_P^+]\subset {\rm med}(Q)
    .
    $$
    Consequently, Lemma~\ref{LemNotInLine} entails that $P=Q$, which contradicts our base assumption. We deduce that $m_P^-=m_P^+$. A similar reasoning entails that $m_Q^-=m_Q^+$. \smallskip 
    
    Then let ${\rm med}(P) \equiv \{m_P\}$ and ${\rm med}(Q)  \equiv\{m_Q\}$ with $m_P\ne m_Q$ (otherwise ${\rm med}(P)={\rm med}(Q)$ and this would again yield a contradiction). Without loss of generality, assume that $m_P<m_Q$. For almost all $\lambda\in (-\infty, m_P)$, eq. (\ref{eq:CdfUni}) entails that $|2F_{\Lambda_P}(\lambda)-1|=|2F_{\Lambda_Q}-1|$ so that, by virtue of the characterization~(\ref{eq:CharacCdfUni}), we have $F_{\Lambda_P}(\lambda)=F_{\Lambda_Q}(\lambda)$. This holds for all $\lambda\in (\infty, m_P)$ by approximation and right-continuity. A similar reasoning entails that $F_{\Lambda_P}=F_{\Lambda_Q}$ on $(m_Q,\infty)$. Right-continuity thus implies that 
    $$
    F_{\Lambda_P}(\lambda) 
    = 
    F_{\Lambda_Q}(\lambda)
    ,\spa 
    \forall\ \lambda \in (-\infty, m_P)\cup [m_Q,\infty)
    .
    $$    
    The characterization (\ref{eq:CharacCdfUni}) entails that
    $$
    F_{\Lambda_Q}(\lambda)
    <
    \frac12 
    <
    F_{\Lambda_P}(\lambda)
    ,\spa 
    \forall\ \lambda\in (m_P, m_Q)
    ,
    $$
    so that $2F_{\Lambda_P}(\lambda)-1 > 0 > 2F_{\Lambda_Q}-1$ over $(m_P, m_Q)$. Now observe that, still on $(m_P,m_Q)$, then $|2F_{\Lambda_P}-1| = 2F_{\Lambda_P}-1$ is monotone non-decreasing and $|2F_{\Lambda_Q}-1|= -(2F_{\Lambda_Q}-1)$ is monotone non-increasing. Recalling from~(\ref{eq:CdfUni}) that $|2F_{\Lambda_P}-1|=|2F_{\Lambda_Q}-1|$ almost every where on $(m_P, m_Q)$ thus entails that $2F_{\Lambda_P}-1$ and $2F_{\Lambda_Q}-1$ must be constant on $(m_P, m_Q)$ with opposite values, hence also on $[m_P, m_Q)$ by right-continuity. In particular, $\Lambda_P$ and $\Lambda_Q$ give no mass to $(m_P,m_Q)$. Let us write 
    $$
    F_{\Lambda_P}(\lambda)
    =
    \frac12 + \alpha 
    ,\quad 
    \text{and}\quad 
    F_{\Lambda_Q}(\lambda) = \frac12 - \alpha
    ,\spa 
    \forall\ \lambda\in [m_P, m_Q)
    ,
    $$
    for some $\alpha>0$; we cannot have $\alpha=0$ for this would imply that $F_{\Lambda_P} = F_{\Lambda_Q}$ over $\R$ to the effect that $\Lambda_P=\Lambda_Q$ hence also $P=Q$, which would contradict the assumption that $P\ne Q$. Now observe that 
    $$
    \lim_{\lambda\stackrel{<}{\to}m_P} F_{\Lambda_P}(\lambda)
    =
    \lim_{\lambda\stackrel{<}{\to}m_P} F_{\Lambda_Q}(\lambda)
    \le 
    F_{\Lambda_Q}(m_P) 
    =
    \frac12-\alpha 
    <
    \frac12 + \alpha 
    =
    F_{\Lambda_P}(m_P)
    .
    $$
    We deduce that $F_{\Lambda_P}$ is discontinuous at $m_P$, hence $m_P$ is an atom of $\Lambda_P$. Similarly,
    $$
    \lim_{\lambda\stackrel{<}{\to}m_Q} F_{\Lambda_Q}(\lambda)
    =
    \frac12 - \alpha 
    <
    \frac12 + \alpha 
    =
    \lim_{\lambda\stackrel{<}{\to}m_Q} F_{\Lambda_P}(\lambda)
    \le 
    F_{\Lambda_P}(m_Q)
    =
    F_{\Lambda_Q}(m_Q)
    ,
    $$
    so that $m_Q$ is an atom of $\Lambda_Q$, which concludes the proof.
\end{Proof}

\begin{Lem}\label{LemTaylorNorm}
    Fix $y\in\HH\sm\{0\}$ and $v\in \HH$ such that $\|v\|=1$ and $\ps{y}{v}=0$. Then, as $r\to 0$ we have
    $$
    \frac{y+rv}{\|y+rv\|}
    =
    \frac{y}{\|y\|}
    % +
    % \frac{r}{\|y\|}\Big({\rm Id} - \frac{yy^T}{\|y\|^2}\Big) v 
    -
    \frac{y}{\|y\|^3} \frac{r^2}{2} 
    % \Big( \frac{\ps{y}{v}v-y}{\|y\|^3}\Big)
    +
    o(r^2) 
    .
    $$
\end{Lem}

\begin{Proof}{Lemma \ref{LemTaylorNorm}}
    Fix an arbitrary $y\in\HH\sm\{0\}$ and $v\in\HH$ such that $y+rv\ne 0$, and let $f(r) = (y+rv)/\|y+rv\|$. Direct computations provide 
    $$
    f'(r) 
    =
    \frac{1}{\|y+rv\|}\Big( v - \Big\langle \frac{y+rv}{\|y+rv\|} , v\Big\rangle \frac{y+rv}{\|y+rv\|} \Big)
    ,
    $$
    and 
    $$
    f''(r) 
    =
    \frac{1}{\|y+rv\|^2}
    \Bigg\{ 
    \Big( 3\big\langle \frac{y+rv}{\|y+rv\|}, v\big\rangle^2 - \|v\|^2 \Big) \frac{y+rv}{\|y+rv\|} - 2\big\langle \frac{y+rv}{\|y+rv\|}, v\big\rangle v 
    \Bigg\}
    .
    $$
    Now, if $\ps{y}{v}=0$ and $\|v\|=1$, a second-order Taylor expansion at $r=0$ yields the conclusion.
\end{Proof}

\begin{Prop}\label{PropPDiffQDim2}
    Let $P$ and $Q$ be Borel probability measures on $\HH$, and assume that $\|\Fpg(x)\|=\|F_Q^\g(x)\|$ for all $x$ in a dense subset $\DD\subset \HH$. If $\dim \HH\ge 2$, then $P \ne Q$ only if there exist $m_P,m_Q\in\HH$ distinct and $\beta\in (1/2,1]$ such that 
    $$
    P
    =
    \beta \delta_{m_P} + (1-\beta) \delta_{m_Q}
    ,\quad 
    \text{and}\quad
    Q
    =
    (1-\beta) \delta_{m_P} + \beta \delta_{m_Q}
    ,
    $$
    where $\delta_x$ denotes the Dirac probability measure at $x$.
\end{Prop}

\begin{Proof}{Proposition \ref{PropPDiffQDim2}}
    A direct computation shows that when $P$ and $Q$ are as in the statement, then $\|\Fpg(x)\|=\|F_Q^\g(x)\|$ for all $\HH\sm\{m_P, m_Q\}$. So, assume that $P\ne Q$ and let us show that $P$ and $Q$ have the form prescribed in the statement. Lemma~\ref{LemPdiffQForm} entails that $P$ and $Q$ are supported on a same line $\LL$, have a unique median $m_P\in\LL$ and $m_Q\in\LL$, respectively, and that there exists a probability measure $\mu$ supported on $\LL\sm [m_P,m_Q]$, where $[m_P,m_Q]$ denotes the line segment in $\HH$ joining $m_P$ and $m_Q$, such that 
    $$
    P
    =
    c\mu + \beta \delta_{m_P} + (1-c-\beta) \delta_{m_Q}
    ,\quad 
    \text{and}\quad
    Q
    =
    c\mu + (1-c-\gamma) \delta_{m_P} + \gamma \delta_{m_Q}
    ,
    $$
    for some $c\ge 0$ and $\beta,\gamma>0$. In what follows, we will show that $c=0$ and $\gamma=\beta$, which will provide the conclusion. For all $x\in\HH$, we have
    $$
    F_P^\g(x) 
    =
    c F_\mu^\g(x) 
    + 
    \beta \frac{x-m_P}{\|x-m_P\|} \Ind{x\ne m_P}
    +
    (1-c-\beta) \frac{x-m_Q}{\|x-m_Q\|} \Ind{x\ne m_Q}
    ,
    $$
    and
    $$
    F_Q^\g(x) 
    =
    c F_\mu^\g(x) 
    + 
    (1-c-\gamma) \frac{x-m_P}{\|x-m_P\|} \Ind{x\ne m_P}
    +
    \gamma \frac{x-m_Q}{\|x-m_Q\|} \Ind{x\ne m_Q}
    .
    $$
    Recall from the beginning of Section \ref{Section:characterization} that $\|\Fpg(x)\|=\|F_Q^\g(x)\|$ for all $x\in\HH$ such that $P[\{x\}]+Q[\{x\}]=0$. In particular, this holds for all $x\in \HH\sm \LL$ since $P$ and $Q$ are supported on $\LL$. Plugging in the expressions for $\Fpg$ and $F_Q^\g$ above, squaring and expanding yields, for all $x\in\HH\sm \LL$
    \begin{equation}\label{eq:IdentityMu}
   c
    \Big\langle 
    F_\mu^\g(x) , \frac{x-m_P}{\|x-m_P\|} - \frac{x-m_Q}{\|x-m_Q\|}
    \Big\rangle 
    =
    (\gamma-\beta)\Big( 1 - \Big\langle \frac{x-m_P}{\|x-m_P\|} , \frac{x-m_Q}{\|x-m_Q\|} \Big\rangle\Big)
    .
    \end{equation}
    Fix $t\in (0,1)$ and let $m_t\equiv m_P + t(m_Q-m_P) \in (m_P,m_Q)$. Since $\dim \HH\ge 2$, fix $v\in\HH$ such that $\ps{v}{m_Q-m_P}=0$ and $\|v\|=1$. Then, for any $r>0$, $x\equiv m_t + rv$ belongs to $\HH\sm \LL$. For such $x$, Lemma~\ref{LemTaylorNorm} entails that, as $r\to 0$,
    $$
    \frac{x-m_P}{\|x-m_P\|} 
    =
    \frac{m_Q-m_P}{\|m_Q-m_P\|}
    -
     \frac{r^2}{2t^2} \frac{m_Q-m_P}{\|m_Q-m_P\|^3} 
     +
    o(r^2)
    .
    $$
    and 
    $$
    \frac{x-m_Q}{\|x-m_Q\|}
    =
    - \frac{m_Q-m_P}{\|m_Q-m_P\|}
    + \frac{r^2}{2(1-t)^2} \frac{m_Q-m_P}{\|m_Q-m_P\|^3} 
    +
    o(r^2)
    .
    $$
    Similarly, for $x=m_t + rv$, we have as $r\to 0$
    $$
    F_\mu^\g(x) 
    =
    \int_\HH \frac{x-z}{\|x-z\|}\, d\mu(z) 
    =
    F_\mu^\g(m_t) - \frac{r^2}{2} \int_\HH \frac{m_t-z}{\|m_t-z\|^3}\, d\mu(z) 
    +
    o(r^2)
    ,
    $$
    also taking note that $m_t-z$ is uniformly bounded away from $0$ when $z$ lies in the support of $\mu$ since $m_t\in (m_P, m_Q)$ and $\mu$ is supported in $\HH\sm [m_P, m_Q]$. Letting $B_t\equiv 1/t^2 + 1/(1-t)^2$, and plugging the previous expansions in (\ref{eq:IdentityMu}) yields 
    \begin{eqnarray*}
    \lefteqn{
    % \hspace{-4mm}
    2c \Big\langle F_\mu^\g(m_t), \frac{m_Q-m_P}{\|m_Q-m_P\|} \Big\rangle 
     - 
    cr^2\Bigg\{ 
    \frac{B_t}{2} \Big\langle F_\mu^\g(m_t),  \frac{m_Q-m_P}{\|m_Q-m_P\|^3} \Big\rangle 
    } 
    \\[2mm]
    &&
    \hspace{54mm}
    +~
    \Big\langle \int_\HH \frac{m_t-z}{\|m_t-z\|^3}\, d\mu(z) , \frac{m_Q-m_P}{\|m_Q-m_P\|} \Big\rangle
    \Bigg\} 
    \\[2mm]
    &&
    =
    2 (\gamma-\beta) 
    - 
     \frac{(\gamma-\beta)B_t}{2\|m_Q-m_P\|^2}~ r^2
    +
    o(r^2)
    .
    \end{eqnarray*}
    It follows that 
    \begin{equation}\label{eq:TaylorId1}
    c \Big\langle F_\mu^\g(m_t), \frac{m_Q-m_P}{\|m_Q-m_P\|} \Big\rangle 
    =
    \gamma-\beta
    ,
    \end{equation}
    and
    \begin{eqnarray}\label{eq:TaylorId2}
    \lefteqn{
    \hspace{-60mm}
    c\Bigg\{ 
    \frac{B_t}{2} \Big\langle F_\mu^\g(m_t),  \frac{m_Q-m_P}{\|m_Q-m_P\|^3} \Big\rangle 
    +~
    \Big\langle \int_\HH \frac{m_t-z}{\|m_t-z\|^3}\, d\mu(z) , \frac{m_Q-m_P}{\|m_Q-m_P\|} \Big\rangle
    \Bigg\} 
    }
    \\[2mm]\nonumber
    \hspace{-40mm}
    =
     \frac{(\gamma-\beta)B_t}{2\|m_Q-m_P\|^2}
     .
    \end{eqnarray}
    Plugging (\ref{eq:TaylorId1}) into (\ref{eq:TaylorId2}) yields
    $$
    c\Big\langle \int_\HH \frac{m_t-z}{\|m_t-z\|^3}\, d\mu(z) , \frac{m_Q-m_P}{\|m_Q-m_P\|} \Big\rangle
    =
    0
    .
    $$
    Since $m_t$ and $z$ lie in $\LL$, then $m_t-z$ is proportional to $m_Q-m_P$, to the effect that the previous display yields 
    $$
    c\int_\HH \frac{m_t-z}{\|m_t-z\|^3}\, d\mu(z)
    =
    0
    .
    $$
    Now assume, \textit{ad absurdum}, that $c\ne 0$, so that 
    $$
    \int_\HH \frac{m_t-z}{\|m_t-z\|^3}\, d\mu(z)=0
    .
    $$
    Define the map $g_z(x) = 1/\|x-z\|$ for $z\in \LL\sm [m_P, m_Q]$ and $x\in [m_P, m_Q]$. Then $g_z$ is concave on $[m_P, m_Q]$ and its directional derivative $(\partial g_z)/(\partial v)$ in any direction $v\in \S$ is given by a constant multiple of $(x-z)/\|x-z\|^3$. Consequently, all directional derivatives of the map
    $$
    G(x) 
    \equiv 
    \int_\HH \frac{1}{\|x-z\|}\, d\mu(z) 
    ,
    $$
    vanish at $m_t$. Since $G$ is well-defined and concave over $(m_P, m_Q)$, we deduce that $G$ admits a (local over the line segment $(m_P, m_Q)$) maximum at $m_t$. Since this holds for all $t\in (0,1)$ 
    % and $G$ is continuous over $(m_P, m_Q)$ (by a direct application of the dominated convergence theorem, since $x$ and $z$ are uniformly separated), 
    we deduce that $G$ is constant over $(m_P, m_Q)$. Recalling that $G(x) = \int_\HH g_z(x)\, d\mu(z)$ and $g_z$ is concave over $(m_P, m_Q)$ as well, then together with the fact that $G$ is constant over $(m_P, m_Q)$ this yields $g_z(m_t) = (1-t) g(m_P) + t g(m_Q)$ for all $t\in (0,1)$ and $\mu$-almost every $z\in \LL\sm [m_P, m_Q]$. The latter rewrites
    $$
    \frac{1}{\|m_P + (1-t)(m_Q-m_P)-z\|} 
    =
    \frac{1-t}{\|m_P-z\|} 
    +
    \frac{t}{\|m_Q-z\|}
    ,\spa 
    \forall\ t\in (0,1)
    .
    $$
    Since there exists no such $z\in\LL\sm [m_P, m_Q]$, we deduce that $c=0$. It follows from~(\ref{eq:IdentityMu}) that $\gamma=\beta$. We have $\beta\ge 1/2$ since $m_P$ is the unique median of $P$, and $\beta\ne 1/2$ since otherwise $P=Q$, which concludes the proof.
\end{Proof}

It remains to establish that $P=Q$ when the spatial cdfs coincide \textit{everywhere}. 

\begin{Prop}\label{PropPEqualQDim1}
    Let $P$ and $Q$ be Borel probability measures on $\HH$. If $\|\Fpg(x)\|=\|F_Q^\g(x)\|$ for all $x\in\HH$, then $P=Q$.
\end{Prop}

\begin{Proof}{Proposition \ref{PropPEqualQDim1}}
    Assume, \textit{ad absurdum}, that $P\ne Q$. Then, Lemma \ref{LemPdiffQForm} entails that $P$ and $Q$ have a unique median, $m_P$ and $m_Q$, respectively, and that $P$ and $Q$ coincide on the Borel subsets of $\HH\sm\{m_P,m_Q\}$. We will show that $P$ and $Q$ also coincide at $m_P$ and $m_Q$, to the effect that $P=Q$, which will bring a contradiction. Fix $x\in\HH$, $v\in \S$, and $(s_k)\subset (0,\infty)$ such that $s_k\to 0$ and $x_k\equiv x+s_k v$ satisfies $P[\{x_k\}]=Q[\{x_k\}]=0$ for all $k$. Then, a direct computation yields 
    $$
    F_P^\g(x_k) 
    \to 
    F_P^\g(x) + vP[\{x\}]
    ,\quad \text{and}\quad 
    F_Q^\g(x_k) 
    \to 
    F_Q^\g(x) + vQ[\{x\}]
    .
    $$
    We deduce that
    $$
    \|F_P^\g(x) + vP[\{x\}]\| 
    =
    \| F_Q^\g(x) + vQ[\{x\}]\| 
    .
    $$
    Squaring and expanding provides
    $$
    \|F_P^\g(x)\|^2 - \|F_Q^\g(x)\|^2 + P[\{x\}]^2 - Q[\{x\}]^2 
    =
    2 \big\langle v , Q[\{x\}]F_Q^\g(x) - P[\{x\}]F_P^\g(x)\big\rangle
    .
    $$
    Since this holds for any $v\in \S$ and the l.h.s. of the previous display is independent of $v$, we deduce that 
    $$
    0
    =
    \|\Fpg(x)\|^2 - \|F_Q^\g(x)\|^2 
    =
    Q[\{x\}]^2 - P[\{x\}]^2
    ,\spa\ \forall\ x\in\HH 
    .
    $$
    In particular, we have $P[\{m_P\}]=Q[\{m_P\}]$ and $P[\{m_Q\}]=Q[\{m_Q\}]$. Consequently, we have $P=Q$, which is a contradiction hence concludes the proof.
\end{Proof}

\section{Proof of Theorem \ref{TheorInvertible}}\label{Section:soothbes}

In this section we show some results regarding the smoothness of the quantile map, obtained as the inverse of the spatial cdf $\Fpg$. We start by establishing that the quantile map is well-defined and continuous with respect to the weak topology $\TT_{\rm weak}$ in $\HH$.
% ; throughout the strong (or norm) topology is denoted $\TT_{\rm strong}$. 
The weak convergence in $\HH$ is denoted as $x_n\rightharpoonup x$, while we denote strong (or norm) convergence by $x_n\to x$. Also recall from Section~\ref{secIntro} that $\B$ stands for the open unit ball in $\HH$ and $\S$ stands for the unit  sphere of $\HH$.\smallskip 

We say that $x_0\in \HH$ is a geometric quantile of order $\alpha\in (0,1)$ in direction $u\in \S$ if
$$ x_0\in \argmin_{x \in \HH} M_{\alpha,u}^P(x), \quad  M_{\alpha,u}^P(x)\equiv\int \big\{ \|z-x\|-\|z\| - \alpha\langle u, x \rangle \big\} \, dP(z).
 $$
If the set of geometric quantiles of order $\alpha \in (0,1)$ and direction $u \in \S$ is a singleton, we denote its unique element by $\Qpg(\alpha u)$. The following results provide sufficient conditions under which $\Qpg$ is well defined and coincides with the inverse of $\Fpg$.

\begin{Prop}\label{PropExistUniq}
    Let $P$ be a Borel probability measure on $\HH$. Fix $\alpha\in (0,1)$ and $u\in \S$. (i) $P$ admits at least one geometric quantile of order $\alpha$ in direction $u$. \linebreak (ii) If $P$ is not supported on a single line of $\HH$, then $P$ admits a unique geometric quantile of order $\alpha$ in direction $u$.
\end{Prop}

\begin{Proof}{\Cref{PropExistUniq}}
    Part (i) and Part (ii) can be proved similarly to Theorem~3.6 and Theorem~2.17 in \cite*{Kem1987}; see also Corollary~2.18 in \cite*{Romon2022} and Theorem~2.1 in \cite*{KonPai2025} with $\rho(t)=t$. 
\end{Proof}

\begin{Prop}\label{PropQuantileInvert}
Let $P$ be a Borel probability measure on $\HH$, and assume that $P$ is non-atomic and not supported on a single line of $\HH$. Then, the quantile map\linebreak $\Qpg : \B\to \HH$ is invertible with inverse~$(\Qpg)^{-1}=\Fpg$. In addition, $\Qpg$ is continuous from $(\mathbb{B},\|\cdot\|)$ to $(\HH,\|\cdot\|)$. 
\end{Prop}
%Arguing as in the proof of Proposition~6.1 in \cite*{KonPai1}, it is easy to prove continuity from $(B(H),\TT_{\rm weak})$ to $(H, \TT_{\rm weak})$ where $\TT_{\rm weak}$ is the weak topology on $H$, but it's not clear whether $\TT_{\rm weak}$ can be replaced by $\|\cdot\|$. 

\begin{Proof}{\Cref{PropQuantileInvert}}
Invertibility follows from (by now) standard arguments in the literature on geometric quantiles: it relies on the fact that $x\in \HH$ is a geometric quantile of order $\alpha\in [0,1)$ in direction $u\in \S$ for $P$ if and only if
\begin{equation}\label{eq:1stOrderCond}
\|\Fpg(x)-\alpha u\| 
\le 
P[\{x\}]
.
\end{equation}
This is a well-known fact when $\HH$ is finite-dimensional (see, e.g., eq. (14) in \cite*{KonStu2025}). When $\HH$ is infinite-dimensional it follows, similarly, from the convexity of the objective function $M_{\alpha,u}^P$ since, as in Theorem~4.14 of \cite*{Kem1987}, global minimizers $x\in \HH$ of~$M_{\alpha,u}^P$ are thus characterized by the first-order condition $(\partial M_{\alpha,u}^P)(x)/(\partial v)\ge 0$ for all $v\in \HH$, where we let 
$$
\frac{\partial M_{\alpha,u}^P}{\partial v}(x)
\equiv 
\lim_{t\downarrow 0} \frac{M_{\alpha,u}^P(x+tv)-M_{\alpha,u}^P(x)}{t}
=
P[\{x\}]\|v\| + \ps{v}{\Fpg(x)-\alpha u}
.
$$
Consequently, the first-order condition indeed reduces to (\ref{eq:1stOrderCond}). Injectivity of $\Fpg$ follows as in the proof of Theorem~5.2 in \cite*{KonStu2025} with $\mathfrak{r}\equiv1$, while the surjectivity of $\Fpg$ follows from the existence of geometric quantiles in Proposition \ref{PropExistUniq}(i) and the non-atomicity of $P$ combined with (\ref{eq:1stOrderCond}). It follows from (\ref{eq:1stOrderCond}) that $\Qpg\circ \Fpg = {\rm Id}_H$ and $\Fpg\circ \Qpg = {\rm Id}_{\mathbb{B}}$, which concludes this part of the proof. \smallskip 

Let us now prove continuity of $\Qpg$. For this purpose, let $(\alpha_k u_k)\subset \B$ and $\alpha u\in\B$ be such that $\alpha_k u_k\to \alpha u$, and let us prove that $\Qpg(\alpha_k u_k)\to \Qpg(\alpha u)$. For this purpose, let us first show that $q_k\equiv \Qpg(\alpha_k u_k)$ is bounded. By definition, we have $M^P_{\alpha_k, u_k}(q_k) \le M^P_{\alpha_k, u_k}(x)$ for any fixed $x\in\HH$. Since $(M^P_{\alpha_k, u_k}(x))_k$ is bounded, then $M^P_{\alpha_k, u_k}(q_k)$ is upper-bounded. Consequently, Proposition~2.1 in \cite*{KonPai2025} and Lemma~S.2.2 in \cite*{KonPai1Supp} (with $\rho(t)\equiv t$ and $\psi(t)\equiv 1$) entail that $(q_k)$ is bounded. We now show, as an intermediary step, that $q_k\wto q\equiv \Qpg(\alpha u)$. Recall that, as a consequence of the Banach-Alaoglu theorem (see, e.g., Theorem~3.29 in \cite*{BrezisFunctional}), the weak topology on any bounded set of $\HH$ is metrizable since $\HH$ is separable. As a metric space, we thus have $q_k\wto q$ if and only if every subsequence of $(q_k)$ admits a further subsequence converging weakly to $q$. By abuse of notation, let us still denote by $(q_k)$ an arbitrary subsequence. Because $(q_k)$ is bounded, it admits a subsequence $(q_{k_\ell})$ converging weakly to some $q_*$. A straightforward adaptation of Lemma~2.1 in \cite*{KonPai2025} entails that $(\alpha u, z)\mapsto M^P_{\alpha,u}(z)$ is lower-semicontinuous on $(\B,\TT_{\rm strong})\times (\HH,\TT_{\rm weak})$. Recalling that 
\begin{equation}\label{eq:MIneq}
M^P_{\alpha_{k_\ell}, u_{k_\ell}}(q_{k_\ell})
\le 
M^P_{\alpha_{k_\ell},u_{k_\ell}}(x)
,
\end{equation}
then taking $\liminf$ on both sides of (\ref{eq:MIneq}) yields 
$$
M^P_{\alpha,u}(q_*)
\le 
\liminf_{\ell\to\infty} M^P_{\alpha_{k_\ell}, u_{k_\ell}}(q_{k_\ell})
\le 
\liminf_{\ell\to\infty} M^P_{\alpha_{k_\ell},u_{k_\ell}}(x)
=
M^P_{\alpha, u}(x)
,
$$
where the last equality follows from weak continuity of $\alpha u\mapsto M^P_{\alpha,u}(x)$. Since the last display holds for all $x\in\HH$, we deduce that $q_*$ is a global minimizer of $M^P_{\alpha,u}$ which, by uniqueness (see Proposition~\ref{PropExistUniq}(ii)) entails that $q_*=\Qpg(\alpha u)$. We deduce that $q_k\wto q$. Let us now show that this weak convergence upgrades to a strong one. Assume, \textit{ad absurdum}, that $(q_k)$ does not converge strongly to $q$. But $(q_k)$ is bounded since it converges weakly. Consequently, up to
% We now turn to strong-to-strong continuity $\Qpg$. Let $(\alpha_k u_k)\subset\mathbb B$
% and $v\in\mathbb B$ be such that $    \|\alpha_k u_k-v\|\to0.$
% Set
% \[
%     q_k\equiv\Qpg(\alpha_k u_k),\qquad x\equiv\Qpg(v).
% \]
% Since $\|v\|<1$, there exists $\rho\in(0,1)$ such that $\|\alpha_k u_k\|\le \rho$ for all
% sufficiently large $n$. Then by Proposition~16.17 in \cite*{Bauschke-Combettes.Springer.2017} 
% $(q_k)$ is bounded. Let $(x_{n_k})$ be an arbitrary subsequence. By boundedness and reflexivity, it
% admits a further subsequence, still denoted $(x_{n_k})$, such that $  x_{n_k}\rightharpoonup x_\ast .$ 
% Since $v_{n_k}\to v$ strongly and $v_{n_k} = \nabla h_P(x_{n_k})=  F_P^g(x_{n_k})$ (recall here that we have shown that $\partial h_P(x_{n_k})$ is a singleton), the
% weak--strong closedness of the graph of the maximal monotone operator
% $\partial h_P$ yields $v= F_P^g(x_\ast)$; see \cite[Proposition~17.31]{Bauschke-Combettes.Springer.2017} 
% Thus, since $x\equiv\Qpg(v)$, by uniqueness we derive   $  q_k\rightharpoonup x .$  We show that the convergence is actually strong. Suppose, toward a contradiction,
% that $q_k$ does not converge strongly to $x$. 
passing to a subsequence, we may
assume that
\[
    \|q_k-q\|\to r>0 .
\]
Since $P$ is non-atomic, $P[\{q_k\}]=P[\{q\}]=0$, and therefore
$\Fpg(q_k)=\alpha_k u_k$ and $\Fpg(q)=\alpha u$ by virtue of~(\ref{eq:1stOrderCond}). Then,
$$
\langle \alpha_k u_k- \alpha u,h_k\rangle
=
\int_\HH
\left\langle
\frac{q_k-z}{\|q_k-z\|}\Ind{z\ne q_k}
-
\frac{q-z}{\|q-z\|}\Ind{z\ne q},
q_k-q
\right\rangle\,dP(z).
$$
On the one hand, the l.h.s. of the last display converges to zero because $\alpha_k u_k\to \alpha u$ strongly and $(q_k-q)$ is
bounded. On the other hand, for every fixed $z\in\HH$ with $z\notin(q_k)$ and $z\ne q$ (hence for $P$-almost all $z\in\HH$), we have 
$$
\Big\langle
\frac{q_k-z}{\|q_k-z\|}
-
\frac{q-z}{\|q-z\|},
q_k-q
\Big\rangle
=
\frac{\|q_k-q\|^2}{\|q_k-z\|} + \frac{\ps{q-z}{q_k-q}}{\|q_k-z\|} - \frac{\ps{q-z}{q_k-q}}{\|q-z\|}
.
$$
Observe that $\|q_k-z\|^2 = \|q_k-q\|^2 + 2\ps{q_k-q}{q-z} + \|q-z\|^2$ converges to $r^2 + \|q-z\|^2$. The dominated convergence theorem thus entails that
$$
0 
=
\liminf_{k\to \infty}\ps{\alpha_k u_k - \alpha u}{q_k-q} 
=
\int_\HH \frac{r^2}{\sqrt{r^2+\|q-z\|^2}}\, dP(z) 
,
$$
a contradiction. We deduce that $q_k\to q$ in $\HH$, which concludes the proof.
% $$
% $h_k\rightharpoonup0$ and $\|h_k\|\to r$. Set $a\equiv q-z\neq0$, so that $q_k-z=a+h_k$. Hence
% \[
% \left\langle
% \frac{q_k-z}{\|q_k-z\|} \Ind{z\ne q_k}
% -
% \frac{q-z}{\|q-z\|}\Ind{z\ne q},
% h_k
% \right\rangle
% =
% \frac{\langle a,h_k\rangle+\|h_k\|^2}{\|a+h_k\|}
% -
% \left\langle \frac{a}{\|a\|},h_k\right\rangle .
% \]
% Since $h_k\rightharpoonup0$, we have $\langle a,h_k\rangle\to0$ and
% $\left\langle a/\|a\|,h_k\right\rangle\to0$. Since also $\|h_k\|\to r$, it follows that
% \[
% \|a+h_k\|^2
% =
% \|a\|^2+2\langle a,h_k\rangle+\|h_k\|^2
% \to
% \|a\|^2+r^2.
% \]
% Therefore,
% \[
% \left\langle
% \frac{q_k-z}{\|q_k-z\|}
% -
% \frac{x-z}{\|x-z\|},
% h_k
% \right\rangle
% \to
% \frac{r^2}{\sqrt{\|x-z\|^2+r^2}} .
% \]
%  Therefore, by the dominated convergence theorem,
% \[
% \begin{aligned}
% 0
% =
% \liminf_{n\to\infty}\langle \alpha_k u_k-v,h_k\rangle=
% \int_\HH
% \frac{r^2}{\sqrt{\|x-z\|^2+r^2}}\,dP(z)
% >0,
% \end{aligned}
% \]
% which is impossible. Hence $r=0$, and therefore $q_k\to x$ strongly. This proves
% that
% \[
%     \Qpg:(\mathbb B,\|\cdot\|)\to(\HH,\|\cdot\|)
% \]
% is continuous.
\end{Proof}

We provide sufficient conditions for the quantile map to be a diffeomorphism between the unit ball and the whole space. For a closed linear subspace $V \subset H$ we denote as $\Pi_V:H\to V$ the orthogonal projection onto $V$. For any Borel probability measure $P$ on $\HH$, we denote the pushforward of $P$ through $\Pi_V$ by $\Pi_V\# P$, \ie the probability measure $(\Pi_V\# P)(A)\equiv P(\Pi_V^{-1}(A))$ for all Borel sets $A\subset H$. For any $p\in [1,\infty)$ and integer $d\ge 1$, we denote by $L^p_\loc(\R^d)$ the collection of measurable maps $u:\R^d\to \R$ such that $\int_K |u(x)|\, dx < \infty$ for all compact sets $K\subset \R^d$. \smallskip

Let $U\subset \HH$ be an open set of $\HH$.  A map $F:U \to \HH$ is said to be of class $C^\ell$ if it is $\ell$-times continuously Fr\'echet differentiable; see p.~46 in \cite*{Deimling1985}.

\begin{Theor}\label{theorem:Manifold}
 Let $P$ be a Borel probability measure on $\HH$ and assume that there exists a linear subspace $V \subset \HH$ of finite dimension~$d\ge 2$ such that $\Pi_V \# P$ admits a (non-identically zero) density $f_V\in L^p_\loc(V)$ with respect to the $d$-dimensional Hausdorff measure on $V$, for some $p\in (d/(d-\ell),\infty]$ and integer $\ell\in [1,d-1]$. Then $\Qpg: \mathbb{B}\to \HH$ defines a $C^\ell$-diffeomorphism with inverse $\Fpg$.  
\end{Theor}

\begin{Proof}{Theorem~\ref{theorem:Manifold}}
    % Let us verify that the assumptions of Proposition \ref{theorem:C-1-hyper} hold. 
% \end{Proof}
% \vspace{3mm}

% \begin{Prop}\label{theorem:C-1-hyper}
%      Let $P\in\PPP(H)$ and assume that $P$ is not supported on a single line of $\HH$. If $P$ satisfies the moment condition 
%      $$
%      \int_{\HH} \frac{1}{\|z-x\|}\, dP(z) 
%      <
%      \infty, 
%      \spa 
%      \forall\ x\in \HH 
%      ,
%      $$
%      % such that for any compact set $\mathcal{K}$ there exist $C,\epsilon\in (0,\infty)$ such that 
%      % \begin{equation}\label{eq:integra-condition}
%      %     \int \frac{1}{\|x-y\|^{1+\epsilon}} dP(x) \leq C <\infty
%      % , \quad\text{for all $y\in \mathcal{K}$}\end{equation}
%      % Assume  assume that $P$  does not give mass to straight lines. 
%      then, for any $\alpha\in (0,1)$, the quantile contour $\mathcal{C}_P(\alpha)$ is a nonempty hypersurface of class $C^1$.
% \end{Prop}

% Comment on the moment condition (non-uniformity)

% \begin{Proof}{Proposition~\ref{theorem:C-1-hyper}}
    % The proof is divided in several steps. First we show that $\Fpg$ is $\mathcal{C}^1$. Next we derive that  $\mathcal{C}_P(\alpha)$ is nonempty.  Then we show that for all $x\in \HH$, there exists $\HH$ with $\|h\|=1$ such that  $ \frac{d}{d t} \big\vert_{t=0}\| \Fpg(x+th) \|^2  \neq 0$. Finally, we conclude by the implicit function theorem. 
    We will first establish that $P$ is non-atomic and not supported on a single line of $\HH$, to the effect that $\Qpg:\mathbb{B}\to \HH$ is a bijection with inverse $(\Qpg)^{-1}=\Fpg$ (Proposition \ref{PropExistUniq}).
    % Notice that the moment condition implies that $P$ is non-atomic. Consequently, Proposition \ref{PropExistUniq} entails that $\Qpg:\mathbb{B}\to H$ is a bijection with inverse $(\Qpg)^{-1}=\Fpg$. 
    In particular, for any $\alpha\in (0,1)$, the quantile contour $\CC_P(\alpha)=\Qpg(\alpha \S)$ is non-empty, and $\CC_P(\alpha)$ is the image under $\Qpg$ of the smooth hypersurface $\alpha \S$. We will then establish that $\Qpg$ is a diffeomorphism of class $C^\ell$, which will yield the result. For this purpose, we will show that $\Fpg$ is $C^\ell$-Fréchet differentiable over $\HH$ with invertible derivative, and we will make use of a `Banach space version' of the inverse function theorem to conclude that $\Qpg$ is $C^\ell$-Fréchet differentiable over $\mathbb{B}$. 

    We only prove the case $\ell=1$. The case $\ell>1$ is similar and can be done as in the proof of Proposition~4.1 in \cite*{Kon2025}. For this purpose, choose a coordinate system on $V$ that realizes the identification $V\simeq \R^d$. For any $x\in \HH$, denote by $\tilde x\in \R^d$ the vector of $\R^d$ corresponding to $\Pi_V(x)\in V$. Since $p>d/(d-1)$, fix $\eta>0$ such that 
      \begin{equation}
      \label{eq:eta}
      \eta < \frac{d-1}{p}\Big(p-\frac{d}{d-1}\Big) 
      .
      \end{equation}
      Fix a bounded subset $E\subset \HH$, and $x\in E$. Since $\Pi_V\# P$ is non-degenerate, we have
    \begin{eqnarray*}
    \lefteqn{
    \int_{\HH} \frac{1}{\|x-z\|^{1+\eta}} dP(z)
    \le
    \int_{\HH} \frac{1}{\|\Pi_V(x-z)\|^{1+\eta}} dP(z)
    }
    \\[2mm]
    &&
    =
    \int_{\HH} \frac{1}{\|\Pi_V(x)-y\|^{1+\eta}} d(\Pi_V\# P)(y)
    =
    \int_{\R^d} \frac{1}{\|\tilde x -y_d\|^{1+\eta}} f_V(y_d)\, dy_d
    .
        \end{eqnarray*} 
    Letting $B(p)$ denote the unit ball in $\R^d$ centered at $p\in\R^d$, then H\"older's inequality entails that 
    \begin{align*}
        \int_{\R^d} \frac{1}{\|\tilde x-y_d\|^{1+\eta}} f_V(y_d)\, dy_d
        &\le 
        1 + \int_{B(\tilde x)}  \frac{1}{\|\tilde x-y_d\|^{1+\eta}}  f_V(y_d)  dy_d 
        \\
        &\le  
        1 + 
        \bigg(
        \int_{B(0)}  \frac{1}{\|y_d\|^{\frac{p(1+\eta)}{p-1}}} dy_d \bigg)^{1-\frac{1}{p}} 
        \bigg( \int_{B(\tilde x)} f_V(y_d)^p dy_d \bigg)^{\frac{1}{p}}.
    \end{align*}
Since $f_V\in L^p_\loc(\R^d)$ and $E$ is bounded, the last integral in the previous display is uniformly bounded over $x\in E$. The integral
$$
\int_{B(0)}  \frac{1}{\|y_d\|^{\frac{p(1+\eta)}{p-1}}} dy_d
$$
is finite if and only if $p(1+\eta)/(p-1) < d$, \ie 
$$
\eta < \frac{d-1}{p}\Big(p-\frac{d}{d-1}\Big) 
. 
$$
Since we chose $\eta>0$ such that (\ref{eq:eta}) holds, we deduce that 
\begin{equation}\label{eq:UI}
\sup_{x\in E} \int_{\HH} \frac{1}{\|x-z\|^{1+\eta}}\, dP(z) 
< 
\infty 
.
\end{equation}
Since this holds for any bounded set $E\subset \HH$, then $P$ is non-atomic. In addition, since $\dim V\ge 2$ and $\Pi_V\#P$ is absolutely continuous with respect to
Lebesgue measure on $V$, $P$ cannot be supported on a single line of $\HH$.  Consequently, Proposition~\ref{PropExistUniq} entails that $\Qpg$ is invertible with inverse $\Fpg$.   
Let us turn to $C^1$-Fréchet differentiability. In view of (\ref{eq:UI}), Proposition 5.5 in \cite*{Romon2022} entails that $\Fpg$ is Fréchet differentiable on $\HH$ with invertible Fréchet derivative $D\Fpg(x): \HH\to \HH$ at any $x\in \HH$ given by 
$$
D\Fpg(x)[h]
\equiv 
\int_{\HH} \frac{1}{\|x-z\|}\Big(h-\frac{\ps{h}{x-z}}{\|x-z\|} \frac{x-z}{\|x-z\|}\Big)\I[z\neq x]\, dP(z) 
,\spa 
\forall\ h\in \HH 
.
$$
% \com{ The moment condition to guarantee that $  \Fpg$ is {\bf continuously} differentiable.}
In addition, the uniform integrability established in (\ref{eq:UI}) entails that $x\mapsto D\Fpg(x)$ is continuous in the strong topology of the operator norm: for any sequence $(x_n)\subset \HH$ and $x\in \HH$ such that $x_n\to x$, we have
$$
\|D\Fpg(x_n)-D\Fpg(x)\|_{\rm op}
\equiv 
\sup_{\|h\|\le 1} \|D\Fpg(x_n)[h]-D\Fpg(x)[h]\|
\to 
0
,\spa 
n\to \infty 
.
$$
Consequently, the inverse function theorem for Banach spaces (see, e.g., Theorem 15.2 in \cite*{Deimling1985}) entails that $\Qpg$ is of class  $C^1$ over $\mathbb{B}$ and the result follows. 
\end{Proof}

The proof of Theorem~\ref{theorem:Manifold} uses the inverse function theorem for Banach spaces applied to $\Fpg$. The Fréchet derivative $D\Fpg(x)$ is invertible by \cite[Proposition~5.5]{Romon2022}, whose proof relies on Neumann series expansions and the fact that small perturbations of invertible bounded linear operators remain invertible.  An alternative approach is to observe that $D\Fpg(x)$ is a compact perturbation of a nonzero multiple of the identity on $\HH$, so that Fredholm theory applies. In particular, Fredholm's alternative implies that $D\Fpg(x)$ is a bounded linear isomorphism provided it is injective. Injectivity follows straightforwardly from the fact that $P$ is not concentrated on a single line (see, e.g., Lemma~S.8.1 in \cite*{KonPai1Supp} with $\rho(t)=t$, or Proposition~5.5(3) in \cite*{Romon2022}), and the result then follows.\smallskip

When $P$ is not supported on a single line of $\HH$, so that quantiles are unique, we define the quantile contours of $P$ as
$$
\CC_P(\alpha)
\equiv 
\big\{ \Qpg(\alpha u) : u\in \S \big\} 
,\spa 
\alpha \in (0,1)
.
$$
When, in addition, $P$ is non-atomic, Proposition \ref{PropExistUniq} entails that $\CC_P(\alpha)$ is given by the $\alpha$-level set of the geometric distribution function, \ie 
$$ 
\CC_P(\alpha)
= 
\{x\in \HH:\|\Fpg(x)\|=\alpha\}
,\spa 
\alpha\in (0,1)
.
$$
\Cref{theorem:Manifold} provides sufficient conditions for the smoothness of the quantile regions. 
\begin{Corol}\label{Corollary:Manifold}
 Let $P$ be a Borel probability measure on $\HH$, and assume that there exists a linear subspace $V \subset H$ of finite dimension~$d\ge 2$ such that $\Pi_V \# P$ admits a density $f_V\in L^p_\loc(V)$ with respect to the $d$-dimensional Hausdorff measure on $V$, for some $p\in (d/(d-\ell),\infty]$ and integer $\ell\in [1,d-1]$. Then, for any $\alpha\in (0,1)$, the quantile contour $\CC_P(\alpha)$ is~$C^\ell$-diffeomorphic to the unit sphere $\S$ of $\HH$. 
\end{Corol}

In particular, if there exist linear subspaces $(V_n)_{n\ge 1}$ of dimension $\dim V_n < \infty$ such that $\Pi_{V_n}\#P$ has a Lebesgue density $f_n\in L^{p_n}_\loc(V_n)$ for some $p_n>1$ with 
\begin{equation}\label{eq:pVn}
\limsup_{n\to \infty}\  \frac{(p_n-1)\dim V_n}{p_n}  \
= 
\infty 
,
\end{equation}
then $\CC_P(\alpha)$ is an infinite-dimensional $C^\infty$-smooth manifold and $C^\infty$-diffeomorphic to $\S$ for any $\alpha\in (0,1)$. Indeed, along a subsequence, we have $\dim V_n (p_n-1)/p_n\to \infty$, so that one can pick a sequence $(\ell_n)\to \infty$ with $\ell_n < \dim V_n (p_n-1)/p_n$ or, equivalently, $p_n > \dim V_n / (\dim V_n- \ell_n)$, to the effect that $\CC_P(\alpha)$ is $C^{\ell_n}$-diffeomorphic to $\S$ for all $n$ and $\alpha \in (0,1)$. Notice that, if $\limsup$ in (\ref{eq:pVn}) is replaced by a limit, the condition $\dim V_n \to \infty$ is necessary for (\ref{eq:pVn}) to hold, to the effect that (\ref{eq:pVn}) is equivalent to $\dim V_n\to \infty$ and $(p_n-1)\dim V_n \to \infty$. In particular if $\dim V_n \to \infty$ and, for some $p>1$, we have $f_n\in L^p_\loc(V_n)$ for all $n$, then the previous conclusion holds.

% \textcolor{red}{Comment on integrability versus regularity as in finite dim. Here, just integrability (decay)}

\begin{Examples} 
After choosing a countable orthonormal basis $\{e_k : k\ge 1\}$ of $\HH$, we may isometrically identify $\HH$ with $(\R^\infty, \|\cdot\|_{\ell^2(\N)})$.
% through 
% $$
% \varphi : H\to (\R^\infty, \|\cdot\|_{\ell^2(\N)}),\ x\mapsto
% \big(\ps{x}{e_k} : k\ge 0\big) 
% ,\spa 
% \forall\ x\in \HH 
% .
% $$
Then, consider the linear subspaces $V_n={\rm span}\{e_1,\ldots, e_n\}$ with corresponding projection $\Pi_n(\sum_{k\ge 1} \alpha_k e_k) = (\alpha_1, \ldots, \alpha_n)$ for any $(\alpha_k)\in\ell^2(\N)$.
\begin{itemize}
    % \item The product measure $P\equiv \prod_{k\ge 0} \NN(0,\sigma_k^2)$ with \textit{positive} coefficients $(\sigma_k)\in \ell^2(\N)$ corresponding to the law of the $\HH$-valued random variable 
    % $$
    % X
    % =
    % \sum_{k\ge 0} 
    % \sigma_k g_k  e_k
    % ,\spa 
    % g_k\stackrel{iid}{\thicksim} \NN(0,1)
    % .
    % $$
    % Then, $\Pi_{ij}\#P$ is the centered Gaussian measure on $\R^2$ with diagonal covariance matrix ${\rm diag}(\sigma_i^2, \sigma_j^2)$.
    \item Consider the centered Gaussian measure $P$ on $\HH$ arising as the law of the centered Gaussian process $\{\W(x) : x\in \HH\}$ with covariance 
    $$
    \E\big[\W(x)\W(y)\big]
    = 
    \ps{x}{Sy}
    ,
    $$
    where $S:H\to H$ is a compact self-adjoint linear operator. The spectral theorem entails that $\HH$ admits an orthonormal basis consisting of eigenvectors of $S$, which we take for $\{e_k : k\ge 1 \}$, \ie 
    $$
    S=\sum_{k\ge 1} \lambda_k^2  e_k\otimes e_k
    ,
    $$ 
    with non-negative eigenvalues $(\lambda_k^2)\to 0$. As a consequence, the probability measure $P$ is the law of the $\HH$-valued random variable 
    $$
    X
    =
    \sum_{k\ge 1} \lambda_k g_k  e_k
    ,\spa 
    g_k\stackrel{iid}{\thicksim} \NN(0,1)
    .
    $$
    Then, $\Pi_n\#P$ is the centered Gaussian measure on $\R^n$ with diagonal covariance matrix $\diag(\lambda_1^2,\ldots, \lambda_n^2)$. In particular, if $S$ is non-degenerate, \ie $\lambda_k>0$ for all $k$, then, for all $n$, $\Pi_n\#P$ admits a density $f_n\in L^\infty(\R^n)$ so that $\CC_P(\alpha)$ is a $C^\infty$-manifold diffeomorphic to $\S$ for any $\alpha\in (0,1)$.
%     $$
%     \Pi_{ij}\#P
%     =
%     \NN\bigg(\begin{bmatrix} 0\\ 0\end{bmatrix}, 
% \begin{bmatrix}
% \lambda_i^2 & 0\\
% 0 & \lambda_j^2
% \end{bmatrix}
% \bigg) 
%     .
%     $$
    % is the centered Gaussian measure on $\R^2$ with diagonal covariance matrix ${\rm diag}(\lambda_i^2, \lambda_j^2)$.
    \item Let $S$ be a covariance operator as in the previous example. Consider centered probability measures $\{\mu_k : k\ge 0\}$ on $\R$ with variance $\lambda_k$, and the probability measure $P$ arising as the law of the $\HH$-valued random variable 
    $$
    X
    =
    \sum_{k\ge 0} g_k e_k 
    ,\spa 
    g_k \thicksim \mu_k
    ,
    $$
    with covariance $S$, \ie $\E[g_i g_j] = \lambda_i^2 \delta_{ij}$,
    where $\delta_{ij}$ stands for the Kronecker symbol. Then, $\Pi_n\#P$ is the probability measure on $\R^n$
    $$
    \Pi_n\#P
    = 
    {\rm Law}(g_1, \ldots, g_n)
    .
    $$
    If the $g_i$'s, \ie the coordinate projections of $X$ in the basis $\{e_k:k\ge 0\}$, are mutually independent, then $\Pi_n\# P$ is simply the product measure $\mu_1\otimes \ldots \otimes \mu_n$. In this case, for any $n\ge 0$, $\Pi_n\# P$ has a density $f_n\in L^{p_n}_\loc(\R^n)$ for some $p_n\in [1,\infty]$ provided $\mu_1,\ldots, \mu_n$ all have a Lebesgue density in $L^{p_n}_\loc(\R)$. In particular if, for some sequence $k_n\to \infty$ such that $k_n/n$ is non-increasing, we have $\mu_n\in L^{1+k_n/n}_\loc(\R)$ for all $n$,   
    % has a density in $L^{1+k_n/n}$, for all $k\ge 0$, we have 
    % $$
    % \mu_k\in L^{1+\delta_k}_\loc(\R)
    % ,
    % \text{for some }
    % \delta_k>0
    % ,
    % $$
    then one can take $p_n=1+k_n/n$ to the effect that $f_n\in L^{p_n}_\loc(\R^n)$ with $(p_n-1)n\to\infty$. Consequently, $\CC_P(\alpha)$ is a $C^\infty$-manifold diffeomorphic to $\S$ for every $\alpha\in (0,1)$.
\end{itemize}
\end{Examples}
% Note that, contrary to Proposition \ref{PropExistUniq}, we do not need to assume in Corollary \ref{Corollary:Manifold} that $P$ is non-atomic and not supported on a single line of $\HH$ since this will be a consequence of the assumption in Corollary~\ref{Corollary:Manifold}. 

\begin{Remark}[Negligibility of the quantile contours]
Some test statistics based on depth functions are consistent under the assumption that the depth contours are negligible with respect to $P$. In $\R^d$, this property holds for most depth functions when $P$ does not assign mass to small sets, for example, when it admits a density with respect to the Lebesgue measure. In infinite-dimensional Hilbert spaces, the Lebesgue measure does not exist, and the most common notion of small sets is given by \emph{Gaussian null sets}, \ie sets that are negligible with respect to any non-degenerate Gaussian measure.
The classes of Gaussian null and \emph{Aronszajn null sets}\footnote{Recall that $E \subset \HH$ is an \emph{Aronszajn null set} (cf.~\cite*{Csrnyei1999AronszajnNA}) if there exists a complete sequence $\{e_i\}_{i \in \N} \subset \HH$ such that $E$ can be written as a countable union of Borel sets $\{E_i\}_{i \ge 1}$ with the property that each $E_i$ is null on every line in the direction $e_i$; that is, for every $i \in \N$ and $a \in \HH$, $\{ t \in \R : a + t\, e_i \in E_i \}| = 0,$
where $|\cdot|$ denotes the one-dimensional Lebesgue measure.} coincide.  
    We notice that under the setting of \Cref{Corollary:Manifold} the quantile contours are Lipschitz hypersurfaces, which implies that they are Gaussian null sets. 
\end{Remark}

\section*{Acknowledgements}
Dimitri Konen acknowledges funding from an ERC Advanced Grant (UKRI G116786).

%%%%%%%%%%%%%%%%%%%%%%%%%%%%%%%%%%%%%%%%%%%%%%%%%%%%%%%%%%%%%
%%                  The Bibliography                       %%
%%                                                         %%
%%  imsart-number.bst  will be used to                     %%
%%  create a .BBL file for submission.                     %%
%%                                                         %%
%%  Note that the displayed Bibliography will not          %%
%%  necessarily be rendered by Latex exactly as specified  %%
%%  in the online Instructions for Authors.                %%
%%                                                         %%
%%  MR numbers will be added by VTeX.                      %%
%%                                                         %%
%%  Use \cite*{...} to cite references in text.             %%
%%                                                         %%
%%%%%%%%%%%%%%%%%%%%%%%%%%%%%%%%%%%%%%%%%%%%%%%%%%%%%%%%%%%%%

\appendix

\section{Gradient flows in a nutshell} 
\label{sec:ReminderFlows} 
In this section, we gather results on the existence and uniqueness, as well as various qualitative properties, of gradient flows associated with a potential $\Phi$. When $\Phi$ satisfies standard regularity and convexity assumptions, then solutions $u$ to the equation 
$$
\dot u(t) 
=
-\nabla \Phi(u(t))
,\spa 
u(0)=x_0
,
$$
exist, where $\dot u(t)$ stands for the time-derivative $du/dt$.  We say that a map $f:[0,\infty)\to\HH$ is \emph{absolutely continuous} on $[0,T]$ if there exists a Borel measurable map $\dot f:[0,T]\to\HH$ such that $ \int_{0}^T\|\dot f(s)\| ds <\infty$ and  
$$
\langle f(t), v\rangle =\langle f(0), v\rangle+\int_0^t \langle \dot f(s), v\rangle  ds 
,\spa 
\text{for all } t\in [0,T] \ \text{and } v\in \HH.
$$
In particular, when this is the case, then $f$ is continuous over $[0,T]$. For a proper convex function $\Phi:\HH\to \R$ (so that, in particular, $\Phi$ is continuous), let the \emph{subgradient} of $\Phi$ at $x$ be defined by
$$
\partial\Phi (x)
\equiv
\bigg\{ z\in\HH : \Phi(x) + \langle z, y-x \rangle \leq \Phi(y),\ \text{for all } y\in\HH
\bigg\} 
.
$$
 We state now the following proposition, which shows the existence of a gradient flow.

\begin{Prop}\label{PropGradFlow} Let $\Phi:\HH\to \R$ be a continuous convex function. Assume that $\inf \Phi > -\infty$. Then, for any $x_0\in \HH$, there exists a unique map $u:[0,\infty)\to \HH$ such that 
	\begin{enumerate}
		\item[(i)] $u$ is continuous and absolutely continuous on each interval $[0,T]$ with $0<T<\infty$;
		\item[(ii)] $-\dot{u}(t)\in \partial \Phi(u(t))$ for almost all $t>0$, and $u(0)=x_0$.
	\end{enumerate}
	The map $u$ satisfies the additional following properties:
	\begin{enumerate}
		\item [(iii)] $\dot{u}\in L^2([0,\infty),\HH)\cap L^\infty([0,\infty),\HH)$. 
		\item[(iv)] The map $t\mapsto \Phi(u(t)) \in \R$ is absolutely continuous on each interval $[0,T]$, with $0<T<\infty$, and a.e.~defined derivative
			$$
			\frac{d}{dt} \Phi(u(t)) 
			=
			-\|\dot{u}(t)\|^2
			 \quad \text{		for a.e.~$t>0$.}
			$$
            Furthermore, $\Phi(u(t))\to \inf \Phi$ as $t\to \infty$.
        \item[(v)] If $\psi$ is a convex and continuous function, and $f\in L^2([0,T],\HH)$ is such that $f(t)\in \partial \psi(u(t))$ for almost all $t\in [0,T]$, then the map $t\mapsto\psi(u(t))$ is absolutely continuous on $[0,T]$ with a.e.~defined derivative
		$$
		\frac{d}{dt} \psi(u(t))
		=
		\langle f(t),\dot{u}(t)  \rangle
		.
		$$
		\item [(vi)] If $\argmin \Phi\neq \emptyset$, then there exists $u_\infty\in \argmin \Phi$ such that $u(t)\rightharpoonup u_\infty$ as $t\uparrow \infty$; %Furthermore, if $\Phi(x)=\Phi(-x)$ for all $x\in \HH$, then $u(t)\to  u_\infty$ in the norm topology as $t\uparrow \infty$. 
	\end{enumerate}
\end{Prop}

\begin{proof}
     The proof of Parts (i)-(iv) can be found in \cite[Theorem~17.2.2]{Attouch.et.al.Variational.SIAM.2014}. The fact that $\Phi(u(t))\to \inf \Phi$ as $t\to \infty$ is shown in Proposition~17.2.7 there  Note that Part~(iii) implies that $u\in W^{1,2}([0,T],\HH)$ for all $T$, so that Part (v) follows from Proposition~17.2.5 there, whereas the last part is proven in Proposition~17.2.11. 
\end{proof}

\bibliographystyle{agsm}
\bibliography{mybib}

@book {BrezisFunctional,
    AUTHOR = {Brezis, Haim},
     TITLE = {Functional analysis, {S}obolev spaces and partial differential
              equations},
    SERIES = {Universitext},
 PUBLISHER = {Springer, New York},
      YEAR = {2011},
     PAGES = {xiv+599},
      ISBN = {978-0-387-70913-0},
}

@article {KonPai23,
    AUTHOR = {Konen, Dimitri and Paindaveine, Davy},
     TITLE = {Spatial quantiles on the hypersphere},
   JOURNAL = {Ann. Statist.},
  FJOURNAL = {The Annals of Statistics},
    VOLUME = {51},
      YEAR = {2023},
    NUMBER = {5},
     PAGES = {2221--2245},
      ISSN = {0090-5364,2168-8966},
}

@article {ChoCha19,
    AUTHOR = {Chowdhury, Joydeep and Chaudhuri, Probal},
     TITLE = {Nonparametric depth and quantile regression for functional
              data},
   JOURNAL = {Bernoulli},
  FJOURNAL = {Bernoulli. Official Journal of the Bernoulli Society for
              Mathematical Statistics and Probability},
    VOLUME = {25},
      YEAR = {2019},
    NUMBER = {1},
     PAGES = {395--423}
}

@article {DaoStuUss24,
    AUTHOR = {Daouia, Abdelaati and Stupfler, Gilles and Usseglio-Carleve,
              Antoine},
     TITLE = {Bias-reduced and variance-corrected asymptotic {G}aussian
              inference about extreme expectiles},
   JOURNAL = {Stat. Comput.},
  FJOURNAL = {Statistics and Computing},
    VOLUME = {34},
      YEAR = {2024},
    NUMBER = {4},
     PAGES = {Paper No. 130, 73},
      ISSN = {0960-3174,1573-1375},
}

@book {Tit1939,
    AUTHOR = {Titchmarsh, E. C.},
     TITLE = {The theory of functions},
      NOTE = {Reprint of the second (1939) edition},
 PUBLISHER = {Oxford University Press, Oxford},
      YEAR = {1958},
     PAGES = {x+454},
}

@book {Deimling1985,
    AUTHOR = {Deimling, Klaus},
     TITLE = {Nonlinear functional analysis},
 PUBLISHER = {Springer-Verlag, Berlin},
      YEAR = {1985},
     PAGES = {xiv+450},
}

@article {KonStu2025,
    AUTHOR = {Konen, Dimitri and Stupfler, Gilles},
     TITLE = {Regularized geometric quantiles and universal linear distribution functionals},
   JOURNAL = {Electron. J. Stat., conditionally accepted},
  FJOURNAL = {Electronic Journal of Statistics},
    VOLUME = {},
      YEAR = {2026},
    NUMBER = {},
     PAGES = {},
      ISSN = {},
       DOI = {},
}

@article {KonPai2025,
    AUTHOR = {Konen, Dimitri and Paindaveine, Davy},
     TITLE = {Existence and breakdown analysis of {M}-quantiles in general
              {H}ilbert spaces},
   JOURNAL = {Electron. J. Stat.},
  FJOURNAL = {Electronic Journal of Statistics},
    VOLUME = {19},
      YEAR = {2025},
    NUMBER = {2},
     PAGES = {5778--5804},
      ISSN = {1935-7524},
}

@article {Hardy1920,
    AUTHOR = {Hardy, G. H.},
     TITLE = {On two theorems of {F}. {C}arlson and {S}. {W}igert},
   JOURNAL = {Acta Math.},
  FJOURNAL = {Acta Mathematica},
    VOLUME = {42},
      YEAR = {1920},
    NUMBER = {1},
     PAGES = {327--339},
      ISSN = {0001-5962,1871-2509},
       DOI = {10.1007/BF02404414},
}

@phdthesis{CarlesonThesis,
	author = {Carlson, F.},
	school = {Facult{\'e} des Sciences d'Upsal},
	title = {Sur une classe de s{\'e}ries de Taylor},
	year = {1914}}

@article {GorKol87,
    AUTHOR = {Gorin, E. A. and Koldobskii, A. L.},
     TITLE = {On potentials of measures in {B}anach spaces},
   JOURNAL = {Sibirsk. Mat. Zh.},
  FJOURNAL = {Akademiya Nauk SSSR. Sibirskoe Otdelenie. Sibirskii
              Matematicheskii Zhurnal},
    VOLUME = {28},
      YEAR = {1987},
    NUMBER = {1},
     PAGES = {65--80, 225},
      ISSN = {0037-4474},
}

@article {KonPai2025a,
    AUTHOR = {Konen, Dimitri and Paindaveine, Davy},
     TITLE = {On the robustness of spatial quantiles},
   JOURNAL = {Ann. Inst. Henri Poincar\'e{} Probab. Stat.},
  FJOURNAL = {Annales de l'Institut Henri Poincar\'e{} Probabilit\'es et
              Statistiques},
    VOLUME = {62},
      YEAR = {2026},
    NUMBER = {1},
     PAGES = {548--581},
      ISSN = {0246-0203,1778-7017},
}

@article{PasReid2022,
	author = {Passeggeri, R. and Reid, N.},
	journal = {Arxiv preprint arXiv:2206.06998},
	title = {A universal robustification procedure},
	year = {2022}}

@article{Romon2022,
	author = {Romon, G.},
	journal = {Arxiv preprint arXiv:2211.00035},
	title = {Statistical properties of approximate geometric quantiles in infinite-dimensional Banach spaces},
	year = {2022}}

@article{Cha1996,
	author = {Chaudhuri, P.},
	journal = {J. Amer. Statist. Assoc.},
	pages = {862--872},
	title = {On a geometric notion of quantiles for multivariate data},
	volume = {91},
	year = {1996}}

@article{Cha2003,
	author = {Chakraborty,B.},
	journal = {J. Statist. Plann. Inference},
	pages = {109--132},
	title = {On multivariate quantile regression},
	volume = {110},
	year = {2003}}

@article{GirStu2015,
	author = {Girard, S. and Stupfler, G.},
	journal = {Extremes},
	pages = {629--663},
	title = {Extreme geometric quantiles in a multivariate regular variation framework},
	volume = {18},
	year = {2015}}

@article{HallEspagne21,
	author = {Hallin, M. and del Barrio, E. and Cuesta-Albertos, J. C. and Matr\'{a}n, C.},
	journal = {Ann. Statist.},
	pages = {1139--1165},
	title = {Distribution and quantile functions, ranks and signs in dimension $d$: a measure transportation approach},
	volume = {49},
	year = {2021}}

@article{Kol1997,
	author = {Koltchinski, V. I.},
	journal = {Ann. Statist.},
	pages = {435--477},
	title = {{M-estimation}, convexity and quantiles},
	volume = {25},
	year = {1997}}

@article{Kon2025,
	author = {Konen, D.},
	journal = {Bernoulli},
        pages = {2077--2104},
        volume = {31},
        number = {3},
	title = {\mbox{PDE} characterization of geometric distribution functions and quantiles},
	year = {2025}}

@article{Kon2025Supp,
	author = {Konen, D.},
	title = {Supplement to ``\mbox{PDE} characterization of geometric distribution functions and quantiles"},
	year = {2025}}

@article{KonPai1,
	author = {Konen, D. and Paindaveine, D.},
	journal = {Bernoulli,},
	pages = {1912--1934},
	title = {Multivariate $\rho$-quantiles: a spatial approach},
	volume = {28},
	year = {2022}}

@article{LopRou1991,
	author = {Lopuha\"a, Hendrik P. and Rousseeuw, Peter J.},
	journal = {Ann. Statist.},
	pages = {229--248},
	title = {Breakdown points of affine equivariant estimators of multivariate location and covariance matrices},
	volume = {19},
	year = {1991}}

@article{Mottonen97,
	author = {Mottonen, J. and Oja, H. and Tienari, J.},
	journal = {Ann. Statist.},
	number = {2},
	pages = {542-552},
	title = {On the efficiency of multivariate spatial sign and rank tests},
	volume = {25},
	year = {1997}
}

@article{PaiVanB2012,
	author = {Paindaveine, Davy and Van Bever, Germain},
	note = {Unpublished},
	title = {Nonparametrically consistent depth-based classifiers},
	year = {2012}}

@inbook{PaiVir2021,
	address = {Cham.},
	author = {Paindaveine, D. and Virta, J.},
	booktitle = {Advances in Contemporary Statistics and Econometrics},
	editor = {Daouia, A. and Ruiz-Gazen, A.},
	pages = {243--259},
	publisher = {Springer},
	title = {On the behavior of extreme $d$-dimensional spatial quantiles under minimal assumptions},
	year = {2021}}

@article{VarZha2000,
	author = {Vardi, Y. and Zhang, C.-H.},
	journal = {Proc. Natl. Acad. Sci. USA},
	number = {4},
	pages = {1423--1426},
	title = {The multivariate {$L_1$}-median and associated data depth},
	volume = {97},
	year = {2000}}

@article{Nag2021,
	author = {Nagy, S.},
	journal = {Statist. Papers},
	pages = {1135-1139},
	title = {Halfspace depth does not characterize probability distributions},
	volume = {62},
	year = {2021}}

@article{KonPai1Supp,
	author = {Konen, D. and Paindaveine, D.},
	title = {Supplement to ``\mbox{Mu}ltivariate $\rho$-quantiles: a spatial Approach"},
	year = {2022}}

@article{Frietal2012,
	author = {Fritz, Heinrich and Filzmoser, Peter and Croux, Christophe},
	journal = {Computational Statistics},
	number = {3},
	pages = {393--410},
	publisher = {Springer},
	title = {A comparison of algorithms for the multivariate {L1}-median},
	volume = {27},
	year = {2012}}

@misc{Kem1987,
	author = {Kemperman, JHB},
	publisher = {North-Holland, Amsterdam},
	title = {The median of a finite measure on a \mbox{Ba}nach space. \mbox{Sta}tistical Data Analysis based on the \mbox{L1-no}rm and related methods, \mbox{Y. Do}dge},
	year = {1987}}

@article{Nag2017,
	author = {Nagy, S.},
	journal = {Statist. Probab. Lett.},
	pages = {373--378},
	title = {Monotonicity properties of spatial depth},
	volume = {129},
	year = {2017}}

@article {PerezSalasSavid-Math-Prog.2021,
    AUTHOR = {P\'erez-Aros, Pedro and Salas, David and Vilches, Emilio},
     TITLE = {Determination of convex functions via subgradients of minimal
              norm},
   JOURNAL = {Math. Program.},
  FJOURNAL = {Mathematical Programming},
    VOLUME = {190},
      YEAR = {2021},
    NUMBER = {1-2},
     PAGES = {561--583},
      ISSN = {0025-5610,1436-4646},
}

@book {Attouch.et.al.Variational.SIAM.2014,
    AUTHOR = {Attouch, Hedy and Buttazzo, Giuseppe and Michaille, G\'erard},
     TITLE = {Variational analysis in {S}obolev and {BV} spaces},
    SERIES = {MOS-SIAM Series on Optimization},
    VOLUME = {17},
   EDITION = {Second},
      NOTE = {Applications to PDEs and optimization},
 PUBLISHER = {Society for Industrial and Applied Mathematics (SIAM),
              Philadelphia, PA; Mathematical Optimization Society,
              Philadelphia, PA},
      YEAR = {2014},
     PAGES = {xii+793},
      ISBN = {978-1-611973-47-1},
}

@book{Bauschke-Combettes.Springer.2017,
    AUTHOR = {Bauschke, Heinz H. and Combettes, Patrick L.},
     TITLE = {Convex analysis and monotone operator theory in {H}ilbert
              spaces},
    SERIES = {CMS Books in Mathematics/Ouvrages de Math\'ematiques de la
              SMC},
   EDITION = {Second},
      NOTE = {With a foreword by H\'edy Attouch},
 PUBLISHER = {Springer, Cham},
      YEAR = {2017},
     PAGES = {xix+619},
      ISBN = {978-3-319-48310-8; 978-3-319-48311-5},
}

@article{Csrnyei1999AronszajnNA,
  title={A{r}onszajn null and {G}aussian null sets coincide},
  author={Marianna Cs{\"o}rnyei},
  journal={Israel J. Math.},
  year={1999},
  volume={111},
  pages={191--201}
}

@article{gonzalezsanz.etl.a.2025monotone,
      title={Monotone measure-transportation maps in Hilbert spaces, with statistical applications}, 
      author={Alberto González-Sanz and Marc Hallin and Bodhisattva Sen},
      year={2025},
journal={To appear in Bernoulli}
}

@article {Lyons.2013.AoP,
    AUTHOR = {Lyons, Russell},
     TITLE = {Distance covariance in metric spaces},
   JOURNAL = {Ann. Probab.},
  FJOURNAL = {The Annals of Probability},
    VOLUME = {41},
      YEAR = {2013},
    NUMBER = {5},
     PAGES = {3284--3305},
      ISSN = {0091-1798,2168-894X},
}

\bigskip
\bigskip

\textsc{Dimitri Konen}

\textsc{Department of Pure Mathematics \& Mathematical Statistics}

\textsc{University of Cambridge}, Cambridge, United Kingdom

Email: dk738@cam.ac.uk

\bigskip

\textsc{Alberto Gonz{\'a}lez-Sanz}

\textsc{Department of Statistics}

\textsc{Columbia University}, New York, United States

Email: alberto.gonzalezsanz@columbia.edu

% %% if your bibliography is in bibtex format, uncomment commands:
% \bibliographystyle{imsart-number} % Style BST file (imsart-number.bst or imsart-nameyear.bst)
% \bibliography{Paper}       % Bibliography file (usually '*.bib')

%% or include bibliography directly:
% \begin{thebibliography}{}
% \bibitem{b1}
% \end{thebibliography}

\end{document}